\documentclass[reqno,11pt]{article}
\usepackage{graphicx, amsmath, amsthm,amssymb}  
\usepackage{tikz}
\usetikzlibrary{matrix,arrows}
\usepgflibrary{fpu}
\usepackage[authormarkup=none,final]{changes}

\definechangesauthor[name={Nils}, color={blue}]{NB}

\usepackage{enumitem}
\setlist{parsep=1pt} 

\RequirePackage{bera}

\input{stochbif.sty}

\input{regstruct_new.sty}

\DeclareRobustCommand{\RobustRSoWVnm}{{\RSoWV_{nm}}}

\def\unit{\vone}

\def\Rplus{\boldsymbol{R^+}\hspace{-.1em}}
\def\normDgamma#1{|\!|\!|#1|\!|\!|}
\def\seminormff#1#2{|\!|\!|#1\,\mathord{;}\,#2|\!|\!|}

\def\KQ{K^Q}

\def\KQhat{\hat K^Q}

\DeclareMathOperator{\vspan}{span}

\DeclareMathOperator{\Id}{Id}           

\def\fraks{\mathfrak{s}}
\def\fraK{\mathfrak{K}}

\def\abss#1{\abs{#1}_\mathfrak{s}}

\def\Tbar{\overline T}

\newcommand{\smallRSV}{\scalebox{0.7}{\RSV}}

\newcommand{\smallRSW}{\scalebox{0.7}{\RSW}}

\DeclareMathOperator{\D}{D}            

\newcommand{\setnm}{\mathfrak{N}}
\newcommand{\sumnm}{\sum_{(n,m)\in\setnm}}
\newcommand{\sumnmarg}[1]{\sum_{\substack{(n,m)\in\setnm\\#1}}}
\newcommand{\supnm}{\sup_{(n,m)\in\setnm}}

\begin{document}


\title{Corrigendum to \lq\lq Regularity structures \\ and renormalisation of
FitzHugh-Nagumo SPDEs \\ in three space dimensions \rq\rq}
\author{Nils Berglund and Christian Kuehn}
\date{Revised version, July 29, 2019}

\maketitle

\begin{abstract}
Lemma 4.8 in the article~\cite{BK2016} contains a mistake, which implies a
weaker regularity estimate than the one stated in Proposition 4.11. This does
not affect the proof of Theorem~2.1, but Theorems~2.2 and~2.3 only follow from
the given proof if either the space dimension $d$ is equal to $2$, or the
nonlinearity $F(U,V)$ is linear in $V$. To fix this problem and provide a proof
of Theorems~2.2 and~2.3 valid in full generality, we consider an alternative
formulation of the fixed-point problem, involving a modified integration
operator with nonlocal singularity and a slightly different regularity
structure. We provide the multilevel Schauder estimates and 
renormalisation-group analysis required for the fixed-point argument in this 
new setting.  
\end{abstract}


\section{Set-up and mistake in the original article~\cite{BK2016}}

In~\cite{BK2016}, we considered FitzHugh--Nagumo-type SPDEs on the torus 
$\T^d$, $d\in\set{2,3}$, of the form
\begin{align}
\nonumber
\partial_t u &= \Delta_x u + F(u,v) + \xi^\eps\;, \\
\partial_t v &=  a_1 u + a_2 v\;,
\label{eq:fhn01} 
\end{align} 
where $F(u,v)$ is a cubic polynomial, $\xi^\eps$ denotes 
mollified space-time white noise, and $a_1, a_2\in \R$ are scalar
parameters (in the case of vectorial $v$, $a_1$ is a vector and $a_2$ is a
square matrix). Duhamel's formula allows us to represent (mild) solutions
of~\eqref{eq:fhn01} on a bounded interval $[0,T]$ as 
\begin{align}
\nonumber
u_t &= \int_0^t S(t-s) \bigbrak{\xi^\eps_s + F(u_s,v_s)} \6s + S(t)u_0\;, \\
v_t &= \int_0^t  Q(t-s) u_s \6s + \e^{ta_2}v_0\;, 
\label{eq:fhn02} 
\end{align}
where $S$ denotes the heat semigroup and $Q(t) := a_1\e^{t a_2}\chi(t)$, where
$\chi:\R_+\to[0,1]$ is a smooth cut-off function supported on $[0,2T]$ such that
$\chi(t)=1$ for all $t\in[0,T]$. 

In~\cite{BK2016}, we used a lift of~\eqref{eq:fhn02} to a regularity structure
of the form 
\begin{align}
\nonumber
U &= (\cK_{\bar\gamma}+R_\gamma\cR)\Rplus \bigbrak{\Xi+F(U,V)} + Gu_0\;, \\
V &= \cK^Q_\gamma\Rplus U + \widehat Qv_0\;, 
\label{eq:fix00} 
\end{align}
where $\cK_{\bar\gamma}$ is the standard lift of the heat kernel 
(cf.~\cite[Sect.~5]{Hairer2014}), and $\cK^Q_\gamma$ is a new operator lifting 
time-convolution with $Q$. 

The problem is that~\cite[Lemma~4.8]{BK2016} is incorrect (it wrongly assumed
translation invariance of the model for space-time white noise). As a
consequence, \cite[Proposition~4.11]{BK2016} does not prove that $\cK^Q_\gamma$
maps $\cD^{\gamma,\eta}$ into itself for any $\gamma\in(0,\eta+2)$. Instead, it
only shows that $\cK^Q_\gamma$ maps $\cD^{\gamma,\eta}$ into
$\cD^{\gamma',\eta}$ for some $\gamma'\leqs\gamma$ that can at best be slightly
less than $1/2$. 

If we look for a fixed point of~\eqref{eq:fix00} with $U\in\cD^{\gamma,\eta}$,
we have in particular to determine the regularity of $F(U,V)$.
Let $\alpha$ be the regularity of the stochastic convolution, that is, 
\begin{equation}
 \alpha = 
 \begin{cases}
  - \kappa & \text{if $d=2$\;,} \\
  -\frac12 - \kappa & \text{if $d=3$\;.}
 \end{cases}
\end{equation} 
Using~\cite[Prop.~6.12]{Hairer2014} and $2\eta+\alpha 
\geqs 3\eta\wedge(\eta+2\alpha)$, we find that $U^3 \in \cD^{\gamma+2\alpha,
3\eta\wedge(\eta+2\alpha)}$, while 
\begin{equation}
 V \in \cD^{\gamma',\eta}\;, 
 \qquad 
 V^2 \in \cD^{\gamma'+\alpha, 2\eta\wedge(\eta+\alpha)}\;, 
 \qquad
 V^3 \in \cD^{\gamma'+2\alpha, 3\eta\wedge(\eta+2\alpha)}\;.
\end{equation} 
This implies that 
\begin{enumerate}
\item 	If $d=2$, then $F(U,V)$ is still in a space of modelled distributions
$\cD^{\gamma'+2\alpha, 3\eta\wedge(\eta+2\alpha)}$ with positive exponent
$\gamma'+2\alpha$. This is sufficient to carry out the fixed-point argument
stated in~\cite[Prop.~6.5]{BK2016}, which relies in particular
on~\cite[Thm.~7.1]{Hairer2014}, that requires this exponent to be positive.

\item 	If $d=3$ and $F(U,V)$ is linear in $V$, then $F(U,V)\in
\cD^{(\gamma+2\alpha)\wedge\gamma', 3\eta\wedge(\eta+2\alpha)}$. Since 
$(\gamma+2\alpha)\wedge\gamma' > 0$, the fixed-point argument again holds.

\item 	If $d=3$ and $F(U,V)$ contains terms in $V^2$ or $V^3$, however, we can
no longer assume that $F(U,V)$ is in a space of modelled distributions with
positive exponent, and we cannot apply~\cite[Thm.~7.1]{Hairer2014}.  
\end{enumerate}

We thus conclude that~\cite[Thm.~2.1]{BK2016}, which concerns the standard
FitzHugh--Nagumo case with $F(U,V)=U+V-U^3$, still follows from the given
proof. Theorems~2.2 and 2.3, however, are only proved if either $d=2$ or $F$
does not contain any terms in $V^2$ or $V^3$.


\section{Corrected results}

We now provide a different argument allowing to prove the results in full 
generality. Consider the system~\eqref{eq:fhn01} on the $3$-dimensional torus, 
for a general cubic nonlinearity of the form
\begin{equation}
 \label{eq:FUV}
  F(u,v) = \alpha_1 u + \alpha_2 v + \beta_1 u^2 + \beta_2 uv + \beta_3 v^2
 + \gamma_1 u^3 + \gamma_2 u^2v + \gamma_3 uv^2 + \gamma_4 v^3\;.
\end{equation}
Its renormalised version is given by 
\begin{align}
\nonumber
\partial_t u^\eps &= \Delta_x u^\eps + \bigbrak{F(u^\eps,v^\eps) + c_0(\eps) + 
c_1(\eps)u^\eps + c_2(\eps)v^\eps} + \xi^\eps\;, \\
\partial_t v^\eps &= a_1 u^\eps + a_2 v^\eps\;,
\label{eq:SPDE_renorm} 
\end{align}
where $\xi^\eps = \varrho_\eps * \xi$ is a mollification of space-time white 
noise, with mollifier $\varrho_\eps(t,x) = \eps^{-5} 
\varrho(\eps^{-2}t,\eps^{-1}x)$ for a compactly supported function $\varrho: 
\R^4\to\R$ of integral $1$. Below we provide a proof of the following result, 
which is in fact a slight generalisation of~\cite[Thm.~2.2]{BK2016}. 

\begin{theorem}
\label{thm:main} 
Assume $u_0 \in \cC^\eta$ for some $\eta>-\frac23$ and $v_0\in\cC^\gamma$ for 
some $\gamma > 1$. Then there exists a choice of constants $c_0(\eps)$, 
$c_1(\eps)$ and $c_2(\eps)$ such that the system~\eqref{eq:SPDE_renorm} 
with initial condition $(u_0,v_0)$ admits a sequence of local solutions 
$(u^\eps,v^\eps)$, converging in probability to a limit $(u,v)$ as $\eps\to0$. 
The limit is independent of the choice of mollifier~$\varrho$.  
\end{theorem}

This result is more general than~\cite[Thm.~2.2]{BK2016} because we do not 
assume that $\gamma_2=0$, even though we are in dimension $d=3$. The 
renormalisation constants $c_i(\eps)$ are given by 
\begin{align}
\nonumber
c_0(\eps) &= -\beta_1\bigbrak{C_1(\eps)+3\gamma_1 C_2(\eps)}\;, \\
\label{eq:c0c1c2} 
c_1(\eps) &= -3\gamma_1\bigbrak{C_1(\eps)+3\gamma_1 C_2(\eps)}\;,\\
\nonumber
c_2(\eps) &= -\gamma_2\bigbrak{C_1(\eps)+3\gamma_1 C_2(\eps)}\;,
\end{align}
where 
\begin{equation}
\label{eq:C1_C2} 
 C_1(\eps) = \int_{\R^4} G_\eps(z)^2 \6z\;, 
 \qquad 
 C_2(\eps) = 2\int_{\R^4} G(z) 
 \biggpar{\int_{\R^4} G_\eps(z_1)G_\eps(z_1-z) \6z_1}^2 \6z\;.
\end{equation} 
Here $G$ denotes the heat kernel in dimension $d=3$, and $G_\eps = 
G*\varrho_\eps$. It is known that $C_1(\eps)$ diverges as $\eps^{-1}$ while 
$C_2(\eps)$ diverges as $\log(\eps^{-1})$. 

An analogous result holds for vectorial variables $v$, in the same way as 
in~\cite[Thm.~2.3]{BK2016}, but without the restriction on $F(u,v)$ having no 
terms in $u^2 v_i$. In that case, $\gamma_2$ and $c_2(\eps)$ become row 
vectors of the same dimension as $v$. Since all arguments are virtually the 
same, we do not present here the details for this situation. 

The main idea  for proving Theorem~\ref{thm:main} is to 
replace~\eqref{eq:fhn02} by another fixed-point equation, which always involves 
convolution in space and time. The price to pay is that this leads to an 
integral kernel with a singularity that is no longer concentrated at the 
origin, but \lq\lq smeared out\rq\rq\ along the time axis. Therefore we need to 
rederive the multilevel Schauder estimates for this type of kernel, which we do 
in Section~\ref{sec:Schauder}. The resulting fixed-point argument is then 
considered in Section~\ref{sec:fixedpoint}, and the effect of renormalisation 
is addressed in Section~\ref{sec:renormalisation}. 


\section{Alternative integral equation}
\label{sec:Schauder} 

There is an alternative to using the fixed-point equation~\eqref{eq:fix00}. 
Indeed, substituting the expression for $u_t$ in~\eqref{eq:fhn02} in the
expression of $v_t$ and rearranging, we find that $v_t$ can also be represented
as
\begin{equation}
 v_t = \int_0^t S^Q(t-s) \bigbrak{\xi^\eps_s + F(u_s,v_s)} \6s 
 + S^Q(t) u_0  + \e^{ta_2}v_0\;, 
\end{equation} 
where 
\begin{equation}
 S^Q(t) = \int_0^t Q(t-s)S(s)\6s\;.
\end{equation} 
Our aim is thus to lift the operation of convolution with $S^Q$ to the
regularity structure, in order to obtain an equivalent fixed-point equation 
of the form 
\begin{align}
\nonumber
U &= (\cK_{\bar\gamma}+R_\gamma\cR)\Rplus \bigbrak{\Xi+F(U,V)} + Gu_0\;, \\
V &= (\cK^Q_{\bar\gamma}+R^Q_\gamma\cR)\Rplus \bigbrak{\Xi+F(U,V)} 
+ G^Qu_0 + \widehat Qv_0\;, 
\label{eq:fix01}
\end{align}
for some suitable kernels $\cK^Q_{\bar\gamma}$ and $R^Q_\gamma$. We already know
that $S$ is represented by convolution with a kernel $G=K+R$. Hence $S^Q$
corresponds to convolution with a kernel $G^Q=K^Q+R^Q$, where the superscript
$Q$ always indicates time-convolution with $Q$. Thus we have to define the lift
$\cK^Q_\gamma$ of $K^Q$ to the regularity structure, meaning that it should map
$\cD^{\gamma,\eta}$ into $\cD^{\bar\gamma,\bar\eta}$ for some suitable
$\bar\gamma$, $\bar\eta$ and satisfy
\begin{equation}
 \cR\cK^Q_\gamma f = K^Q * \cR f\;.
\end{equation} 


\subsection{Decomposition of the kernel}

The difficulty is that since $K^Q$ is obtained by convolution in time of $K$
with $Q$, its singularity is no longer concentrated at the origin, but is
\lq\lq smeared out\rq\rq\ along the time axis. In fact, we have the following
decomposition result replacing~\cite[Assumption~5.1]{Hairer2014}. 
Note that here and below, we write $z=(t,x)$ for space-time 
points.

\begin{prop}
\label{prop:KQnm} 
Assume $Q$ is supported on $[0,2T]$ for a given $T>0$, fix a scaling
$\fraks=(\fraks_0,\fraks_1,\dots,\fraks_d)$, and let $K$ be a regularizing
kernel of order $\beta$ (cf.~\cite[Ass.~5.1]{Hairer2014}). The kernel $K^Q$ 
obtained by convoluting $Q$ and $K$ in time can be decomposed as 
\begin{equation}
 K^Q(z) = \sum_{(n,m)\in\mathfrak{N}} K^Q_{nm}(z)\;,
\end{equation} 
where $\mathfrak{N} = \setsuch{(n,m)\in\Z^2}{n\geqs0, -1\leqs m \leqs
1+2T2^{\fraks_{0}n}}$ 
and the $K^Q_{nm}$ have the following properties.
\begin{itemize}
\item 	Let $h_{nm} = (m2^{-\fraks_0n},0)$. 
For all $n,m$, $K^Q_{nm}$ is supported on the ball
\begin{equation}
\label{eq:KQnm_1} 
 \bigsetsuch{z\in\R^{d+1}}
 {\norm{z - h_{nm}}_\fraks \leqs (1+2^{1/\fraks_0})2^{-n}}\;.
\end{equation} 

\item 	For any multiindex $k$, there exists a constant $C_Q$ such that 
\begin{equation}
\label{eq:KQnm_2} 
 \bigabs{\D^k K^Q_{nm}(z)} \leqs C_Q 2^{(\abs{\fraks}-\fraks_0-\beta+\abss{k})n}
\end{equation} 
holds uniformly over all $(n, m)\in\mathfrak{N}$ and all $z\in\R^{d+1}$. 

\item 	For any two multiindices $k$ and $\ell$, there exists a constant $C_Q$
such that 
\begin{equation}
\label{eq:KQnm_3} 
 \biggabs{\int_{\R^{d+1}}z^\ell \D^k K^Q_{nm}(z)\6z} 
 \leqs C_Q 2^{-(\beta+\fraks_0)n}
\end{equation} 
holds uniformly over all $(n, m)\in\mathfrak{N}$.
\end{itemize}
\end{prop}

We give the proof in Appendix~\ref{app:KQnm}. 
Note the extra $\fraks_0$ in the bound~\eqref{eq:KQnm_2}, which compensates the
fact that $m$ takes of the order of $2^{\fraks_0 n}$ values. 

\begin{remark}
We only need these results in the case $\beta=2$, and for the parabolic scaling
$\fraks=(2,1,1,1)$. However, since there is no difficulty in dealing with this
more general setting, we may as well do so here.   
\end{remark}


\subsection{Extension of the regularity structure}

In order to lift convolution with $K^Q$ to the regularity structure, 
\added[id=NB]{a 
natural idea is to} enlarge the model space of the Allen--Cahn equation
(cf.~\cite[Sec.~3 and Table~1]{BK2016}) by adding new elements of the form
$\cI^Q(\tau)$ whenever $\abss{\tau}\notin\Z$. By convention, $\cI^Q(\tau)$
then has homogeneity $\abss{\cI^Q(\tau)} = \abss{\tau} + \beta$. 

In order to extend the model, \added[id=NB]{one can then try to proceed} as 
in~\cite[Sect.~5]{Hairer2014} by first introducing functions 
\begin{equation}
\label{eq:def_JQ} 
 \cJ^Q(z)\tau = \sum_{\abss{k} < \alpha+\beta} \frac{X^k}{k!}
 \sumnm \bigpscal{\Pi_z\tau}{\D^k K^Q_{nm}(z-\cdot)}
\end{equation} 
where $\alpha = \abss{\tau}$. Then the model is formally given by  
\begin{equation}
\label{eq:def_modelIQ} 
 (\Pi_z\cI^Q\tau)(\bar z) 
 = \bigpscal{\Pi_z\tau}{K^Q(\bar z-\cdot)} 
 - (\Pi_z\cJ^Q(z)\tau)(\bar z)\;.
\end{equation} 
The precise formulation of this relation is that for any test function $\psi$, 
\begin{equation}
\label{eq:def2_model_IQ} 
 \bigpscal{\Pi_z\cI^Q\tau}{\psi} 
 = \sumnm \int_{\R^{d+1}} \bigpscal{\Pi_z\tau}{K^{Q;\alpha}_{nm;z\bar z}}
\psi(\bar z)\6\bar z\;, 
\end{equation} 
where
\begin{equation}
\label{eq:def_KQalphanm} 
 K^{Q;\alpha}_{nm;z\bar z}(z')
 = K^Q_{nm}(\bar z-z') - \sum_{\abss{k} < \alpha+\beta}
 \frac{(\bar z-z)^k}{k!} \D^kK^Q_{nm}(z-z')\;.
\end{equation}
We still need to verify that all these definitions make sense for the new
kernel. We can however exploit the fact that in practice, we will only need to
apply this construction to symbols $\tau$ \added[id=NB]{whose model does not 
depend on 
the reference time in the following sense.}
\added[id=NB]{%
\begin{definition}
We say that the model $\Pi\tau$ is base-time independent if 
\begin{equation}
 (\Pi_{z+h}\tau)(\bar z) = (\Pi_z\tau)(\bar z)
\end{equation} 
holds for all $z,\bar z\in\R^{d+1}$ and all time shifts $h=(h_0,0) \in 
\R\times\R^d$. 
\end{definition}
}%
\begin{lemma}
\label{lem:modelIQ} 
Assume that $\tau\in T_\alpha$ \added[id=NB]{has a base-time independent model} 
and that 
$\alpha+\beta\not\in\N$. Then the series 
in~\eqref{eq:def_JQ} and~\eqref{eq:def2_model_IQ} are absolutely convergent. 
Furthermore, 
\begin{equation}
\label{eq:modelIQ} 
 \bigabs{\bigpscal{\Pi_z\cI^Q\tau}{\psi^\lambda_z}} 
 \lesssim \lambda^{\alpha+\beta} \norm{\Pi}_{\alpha;\fraK_z}
\end{equation} 
holds uniformly over $z\in\R^{d+1}$ and $\lambda\in(0,1]$, where
$\psi^\lambda_z(\bar z) = \cS^\lambda_{\fraks,z}\psi(\bar z)$ and 
$\fraK_z$ is the ball of radius $2$ centred in~$z$.
Here $\cS^\lambda_{\fraks,z}\psi(\bar z_0,\dots,\bar z_d)
= \lambda^{-\abs{\fraks}}\psi(\lambda^{-\fraks_0}(\bar z_0-z_0), 
\dots, \lambda^{-\fraks_d}(\bar z_d-z_d))$.
\end{lemma}

The proof of this result is very similar to the proof
of~\cite[Lem.~5.19]{Hairer2014}, but there are a few differences due to the
nonlocal singularity of $K^Q$ which we explain in Appendix~\ref{app:model}. The
constant in~\eqref{eq:modelIQ} does not depend on $\norm{\Gamma}$ owing to the
fact that $\Pi$ is base-\added[id=NB]{time} independent. 

\added[id=NB]{%
\begin{remark}
\label{rem:base-time-indep} 
In our particular case, the canonical model of the following symbols is 
base-time independent: 
\begin{equation}
\label{eq:list_base_time_indep} 
\Xi, \RSI, \RSoI, \RSV, \RSVo, \RSVoo, \RSW, \RSWo, \RSWoo, \RSWooo, \unit\;.
\end{equation} 
Here $\RSoI := \cI^Q(\Xi)$ has homogeneity $\abss{\Xi}+2$, and we employ the 
usual notation and additivity rule of homogeneities for products. 
In fact, the canonical model is completely independent of the base point for 
these symbols, so that any translation, not only in the time direction, has no 
effect. Indeed, $(\Pi^\eps_z\Xi)(\bar z) = \xi^\eps(\bar z)$ does not depend on 
$z$, and neither does, for instance, 
\begin{equation}
 (\Pi^\eps_z\RSoI)(\bar z) = \int \xi^\eps(z') \KQ(\bar z-z')\6z'
\end{equation} 
owing to the fact that $\abss{\Xi}+2$ is strictly negative, so that the sum 
in~\eqref{eq:def_JQ} is empty. A similar argument holds for the other symbols 
in the list~\eqref{eq:list_base_time_indep}. 
In addition, the canonical model is base-time independent for monomials of the 
form $X_i$, $i\in\set{1,2,3}$, since $(\Pi^\eps_zX_i)(\bar z) = 
\bar z_i - z_i$, though it does depend on the spatial part of the base point. 
By contrast, the model of $X_0$ is not base-time independent, and neither are 
symbols such as $\RSoIW = \cI^Q(\RSWo)$, which have positive regularity (see 
also Remark~\ref{rem:Gamma_zz} below).
\end{remark}%
}%

In order to also extend the structure group, we first extend the coproduct via 
\begin{equation}
\label{eq:DeltaIQ} 
 \Delta(\cI^Q\tau) = (\cI^Q\otimes\Id) \Delta(\tau) 
 + \sum_{\abss{k+\ell} < \alpha+\beta} 
 \frac{X^k}{k!} \otimes \frac{X^\ell}{\ell!} \cJ^Q_{k+\ell}\tau
\end{equation} 
where the $\cJ^Q_{k+\ell}\tau$ are new symbols satisfying 
\begin{equation}
 \bigpscal{f_z}{\cJ^Q_\ell\tau} = 
 - \bigpscal{\Pi_z\tau}{\D^\ell K^Q(z-\cdot)}\;.
\end{equation} 
Recall that the $f_z$ are linear forms allowing to define the 
structure group by setting $\Gamma_{z\bar z} = F_z^{-1}F_{\bar z}$, where
\begin{equation}
 F_z \tau = (\Id\otimes f_z)\Delta\tau\;.
\end{equation} 
In the particular case $\tau=\Xi$, we obtain that $\cI^Q\tau =: \RSoI$ satisfies
$\Delta(\RSoI) = \RSoI \otimes \unit$ and thus 
\begin{equation}
 F_z \RSoI = \RSoI\;, \qquad 
 \Gamma_{z\bar z} \RSoI = \RSoI\;.
\end{equation} 
In what follows, it will be useful to have explicit expressions for the action
of the structure group on such monomials. Such an expression is provided by the
next result, proved in Appendix~\ref{app:Gammazz}.

\begin{lemma}
\label{lem:Gammazz} 
Assume that $\tau\in T_\alpha$ has a base-\added[id=NB]{time} independent model
$\Pi_z\tau$\deleted[id=NB]{$=\Pi\tau$} and satisfies 
$\Delta(\tau)=\tau\otimes\unit$. Then the
structure group acts via 
\begin{equation}
\Gamma_{z\bar z}\cI^Q\tau = \cI^Q\tau + 
\sum_{\abss{k}<\alpha+\beta}
\frac{X^k}{k!} 
\biggbrak{
 \chi_\tau^k(z) - \sum_{\abss{\ell} < \alpha+\beta-\abss{k}}
\frac{(z-\bar z)^\ell}{\ell!} \chi_\tau^{k+\ell}(\bar z)
}
\end{equation}
where $\chi_\tau^k(z) = \sumnm\pscal{\Pi\added[id=NB]{_z}\tau}{\D^k 
K^Q_{nm}(z-\cdot)}$.
\end{lemma}

\begin{remark}
\label{rem:Gamma_zz} 
This result illustrates the fact that~\cite[Lem.~4.8]{BK2016} is incorrect in
general. For instance, in the case $\tau=\RSW$ we obtain
\begin{equation}
 \Gamma_{z\bar z} \RSoIW = \RSoIW + \bigbrak{\chi^0_{\smallRSW}(z) -
\chi^0_{\smallRSW}(\bar z)}
\unit\;.
\label{eq:compute_Gamma2} 
\end{equation} 
Since $\chi^0_{\smallRSW}(z+h) - \chi^0_{\smallRSW}(\bar z+h) \neq
\chi^0_{\smallRSW}(z) - \chi^0_{\smallRSW}(\bar z)$, the operator
$\Gamma_{z\bar
z}$ is indeed not translation invariant. 
\added[id=NB]{An even simpler example of the incorrectness 
of~\cite[Lem.~4.8]{BK2016} 
occurs for the canonical model of $\Xi$, since $(\Pi^\eps_z\Xi)(\bar z) = 
\xi^\eps(\bar z)$ is constant with respect to $z$, not with respect to 
$\bar z$.}
\end{remark}


\subsection{Lifting the convolution operator}

Following the strategy in~\cite[Section~5]{Hairer2014}, it is natural to
look for a lift of the operation of convolution with $K^Q$ given for
$f\in\cD^{\gamma,\eta}$ by 
\begin{equation}
\label{eq:KQgamma_standard} 
 (\cK^Q_\gamma f)(z) = 
 \cI^Q f(z) + \cJ^Q(z) f(z) + (\cN^Q_\gamma f)(z)\;,
\end{equation} 
where the nonlocal operator $\cN^Q_\gamma$ is defined by 
\begin{equation}
\label{eq:NQ1} 
 (\cN^Q_\gamma f)(z) 
 = \sum_{\abss{k} < \gamma+\beta} \frac{X^k}{k!} \sumnm
 \bigpscal{\cR f - \Pi_zf(z)}{\D^k K^Q_{nm}(z-\cdot)}\;.
\end{equation} 
The problem with this definition is that in general, if $f$ is defined on a
sector of regularity $\alpha$, we can only prove a bound of the form 
\begin{equation}
\label{eq:weakbound} 
 \bigabs{\bigpscal{\cR f - \Pi_zf(z)}{\D^k K^Q_{nm}(z-\cdot)}}
 \lesssim 2^{(\abss{k}-\fraks_0-\alpha-\beta)n}\;,
\end{equation} 
instead of $2^{(\abss{k}-\gamma-\beta)n}$ as
in~\cite[Eq.~(5.42)]{Hairer2014}. The reason for this weaker bound is that in
general $K^Q_{nm}(z-\cdot)$ is not supported near the origin, so that
shifting the model as in the proof of~\cite[Lem.~5.18]{Hairer2014} produces an
additional factor of order $\norm{h_{nm}}_\fraks^{\gamma-\alpha}$, which can
have order $1$ instead of order $2^{-(\gamma-\alpha)n}$ as in that Lemma.

The bound~\eqref{eq:weakbound} proves convergence of the sum in~\eqref{eq:NQ1}
only for $\abss{k} < \alpha + \beta$. If for instance $f(z) = a(z)\RSV$, in
dimension $d=3$ the sector has regularity $\alpha=-1-2\kappa$, and thus only the
term with $k=0$ is well-defined. Restricting the sum over $k$ to only the term
$k=0$, however, results in $\cK^Q_\gamma f$ belonging only to some
$\cD^{\bar\gamma,\bar\eta}$ with $\bar\gamma < 1$, which is not sufficient to
carry out the fixed-point argument for a general cubic $F(U,V)$ 
for $d=3$.

A way out of this situation is to work with shift operators. 
Define, for any $h\in\R^{d+1}$, an operator $T_h:
\cC^\alpha_\fraks \to \cC^\alpha_\fraks$ by 
\begin{equation}
 \pscal{T_h\upsilon}{\psi} = \pscal{\upsilon}{\psi(\cdot-h)}
\end{equation} 
for any test function $\psi$. In case $\upsilon$ is a function, this amounts to
setting $T_h\upsilon(z) = \upsilon(z+h)$. 
\added[id=NB]{We define a shifted model $T_h\Pi =: \Pi^h$ by 
\begin{equation}
 \Pi^h_z\tau(\bar z) = \Pi_{z+h}\tau(\bar z+h)
 \qquad\forall z,\bar z \in\R^{d+1}
\end{equation} 
(which should be interpreted as $\Pi^h_z\tau = T_h(\Pi_{z+h}\tau)$ if 
$\Pi_z\tau$ is a distribution).}
Assume we can define, on some sector of $\cD^{\gamma,\eta}(\Pi)$, a
map $\cT_h$ taking values in $\cD^{\gamma,\eta}(\Pi^h)$ and satisfying
\begin{equation}
 \label{eq:cTh}
 \cR^h\cT_h = T_h\cR\;,
\end{equation} 
where $\cR^h$ is the reconstruction operator on $\cD^{\gamma,\eta}(\Pi^h)$. 
\added[id=NB]{If $f$ belongs to a function-like sector,~\eqref{eq:cTh} is 
equivalent to 
\begin{equation}
 (\cR^h\cT_h f)(z) = (\cR f)(z+h)
\end{equation}
where $(\cR f)(z+h) = (\Pi_{z+h}f(z+h))(z+h) = (\Pi^h_z f(z+h))(z)$.}

Setting $\cR^{nm} := \cR^{h_{nm}}$ we have 
\begin{equation}
 \bigpscal{\cR f}{\D^k K^Q_{nm}(z-\cdot)} 
 = \bigpscal{\cR^{nm} f_{nm}}{\D^k \KQhat_{nm}(z-\cdot)}
\end{equation} 
where $f_{nm} = \cT_{h_{nm}}f$ and 
\begin{equation}
\label{eq:KQhat} 
 \KQhat_{nm}(z) = K^Q_{nm}(z + h_{nm})
\end{equation} 
is a shifted kernel, supported in a ball of radius of order $2^{-n}$ around the
origin. Finally, let $\Pi^{nm} = \Pi^{h_{nm}} = T_{h_{nm}}\Pi$ denote the 
time-shifted models, and assume that for each $h_{nm}$, we can define an 
operator $\cK^Q_{\gamma,nm}$ from $\cD^{\gamma,\eta}(\Pi^{nm})$ to
$\cD^{\gamma+\beta,\bar\eta}(\Pi)$ satisfying 
\begin{equation}
\label{eq:conv_KQnm} 
 \cR\cK^Q_{\gamma,nm} = \KQhat_{nm} * \cR^{nm}\;.
\end{equation} 
Then the operator 
\begin{equation}
 \cK^Q_{\gamma} = \sumnm \cK^Q_{\gamma,nm}\cT_{h_{nm}} 
\end{equation} 
maps $\cD^{\gamma,\eta}(\Pi)$ into $\cD^{\gamma+\beta,\bar\eta}(\Pi)$
and satisfies the required identity $\cR\cK^Q_\gamma = K^Q*\cR$. 

The property~\eqref{eq:conv_KQnm} can be achieved by defining
$\cK^Q_{\gamma,nm}$ as in~\eqref{eq:KQgamma_standard}, but replacing the model,
kernel and reconstruction operator in~\eqref{eq:NQ1} and~\eqref{eq:def_JQ} by
their shifted versions. This has the advantage of improving the
bound~\eqref{eq:weakbound}, since the kernel $\KQhat_{nm}$ is now supported near
the origin. A drawback is that this forces us to introduce a countable infinity
of new symbols $\cI^Q_{nm}\tau$, for $\tau$ in the sector under consideration.
We will now show that in the case of FitzHugh--Nagumo-type SPDEs of the
form~\eqref{eq:fhn01}, one can indeed construct a shift map
realising~\eqref{eq:cTh} on a specific sector of negative homogeneity. Then we
will check that the introduction of infinitely many new symbols does not pose a
problem for the renormalisation procedure. 


\subsection{Multilevel Schauder estimates for FitzHugh--Nagumo-type SPDEs}

We now particularise to the FitzHugh--Nagumo-type SPDE~\eqref{eq:fhn01} in
dimension $d=3$. We consider modelled distributions in
$\cD^{\gamma,\eta}$ of the form 
\begin{align}
\nonumber
 f(z) &= \sum_{\tau\in\cF_1} c_\tau \tau 
 + \sum_{\tau\in\cF_2} a_\tau(z) \tau 
 + \varphi(z)\unit
 + \sum_{\tau\in\cF_3} a_\tau(z) \tau \\
 &=: f_1(z) + f_2(z) + \varphi(z)\unit + f_3(z)\;,
\label{eq:f123} 
\end{align} 
where
\begin{align}
\nonumber
\cF_1 &= \set{\RSW,\RSWo,\RSWoo,\RSWooo}\;, \\
\cF_2 &= \set{\RSI,\RSoI,\RSV,\RSVo,\RSVoo,
X_i\RSV,X_i\RSVo,X_i\RSVoo\colon i\in\set{1,2,3}}\;,
\label{eq:tauset_KQ} 
\end{align}
and $\cF_3$ is such that any $\tau\in\cF_3$ satisfies the diagonal identity
\begin{equation}
\label{eq:diag_ident} 
\lim_{\lambda\to0} \pscal{\Pi_z\tau}{\psi^\lambda_z} = 0\;.
\end{equation} 
The reason why we only include polynomial elements $X_i$ in the 
spatial directions in $\cF_2$ is that owing to the polynomial scaling, 
$\abs{X_0}_\fraks=2$ and thus $\abs{X_0{\protect\RSV}}_\fraks>0$. 
By linearity, we may define separately the action of $\cK^Q_\gamma$ on $f_1$,
$f_2$, $\varphi\unit$, and $f_3$. In the case of $f_1$ and $\varphi\unit$, we
use the standard definition~\eqref{eq:KQgamma_standard}, which takes here the
form
\begin{align}
\label{eq:def_KQf1} 
\cK^Q_\gamma f_1(z) &= \sum_{\tau\in\cF_1} 
c_\tau\bigbrak{\cI^Q\tau + \chi^0_\tau(z)\unit}\;, \\
\cK^Q_\gamma \varphi\unit(z) &= 
\sum_{\abss{k} < \gamma+\beta} \frac{X^k}{k!}
\pscal{ \varphi}{\D^k K^Q(z-\cdot)}
\;.
\label{eq:def_KQf1phi} 
\end{align}
Here we have set $\cN_\gamma f_1 = 0$, since we may choose 
$\cR f_1 = \Pi_z f_1(z) = \sum_{\tau\in\cF_1}c_\tau\Pi\tau$, owing to the fact
that $f_1$ does not depend on $z$.
Furthermore, we have used the fact that thanks 
to the vanishing-moments condition, $\cJ^Q(z)\unit=0$ and 
$\pscal{\Pi_z\unit}{\D^k K^Q(z-\cdot)}=0$.
For $f_3$ we simply set 
\begin{equation}
 \cK^Q_\gamma f_3(z) = 0\;,
\end{equation} 
which is allowed thanks to the diagonal identity~\eqref{eq:diag_ident}.

It thus remains to define $\cK^Q_\gamma f_2(z)$. Here we use the procedure
based on shift operators, as outlined above. Owing to the fact that 
the only polynomial terms $X_i$ occurring in $\cF_2$ are purely 
spatial, all $\tau\in\cF_2$ \added[id=NB]{are base-time independent 
(cf.~Remark~\ref{rem:base-time-indep})}. As a consequence, one can check that 
the map $\cT_{h_{nm}}$ can be realised by 
\begin{equation}
 \cT_{h_{nm}} f_2(z) = 
\sum_{\tau\in\cF_2} a_\tau(z+h_{nm})\tau\;.
\end{equation} 
In this way, we obtain 
\begin{align}
\nonumber
 \cK^Q_\gamma f_2(z) = \sumnm \biggl\{ &\sum_{\tau\in\cF_2} 
 a_\tau(z+h_{nm}) \biggbrak{\cI^Q_{nm}\tau 
 + \sum_{\abss{k} < \abss{\tau}+\beta} \frac{X^k}{k!} 
 \hat\chi^k_{\tau,nm}(z)} \\
 &{}+ \sum_{\abss{k} < \gamma+\beta} \frac{X^k}{k!} 
 b^k_{nm}(z) \biggr\}\;, 
 \label{eq:def_KQf2} 
\end{align}
where
\begin{align}
\nonumber
\hat\chi^k_{\tau,nm}(z) &=
\bigpscal{\Pi_z^{nm}\tau}{\D^k\KQhat_{nm}(z-\cdot)}\;,
\\
b^k_{nm}(z) &= \bigpscal{\cR^{nm}f_{nm} - \Pi^{nm}_z
f_{nm}(z)}{\D^k\KQhat_{nm}(z-\cdot)}\;.
 \label{eq:def_bknm} 
\end{align}
Furthermore, the $\cI^Q_{nm}\tau$ are new symbols with model 
\begin{equation}
\label{eq:model_IQnm} 
 (\Pi_z\cI^Q_{nm}\tau)(\bar z) 
 = \hat\chi^0_{\tau,nm}(\bar z) 
 - \sum_{\abss{k} < \abss{\tau}+\beta} \frac{(\bar z-z)^k}{k!}
\hat\chi^k_{\tau,nm}(z)\;.
\end{equation} 
Since $T_{\abss{\tau}+\beta}$ is now infinite-dimensional for $\tau\in\cF_2$, 
the choice of norm on these subspaces matters, and we choose it to be the 
supremum norm. \added[id=NB]{More precisely, writing $\alpha=\abss{\tau}+\beta$, 
we have
\begin{equation}
 g = \sumnm \sum_{\tau\in\cF_2} c_{\tau,nm} \cI^Q_{nm}\tau 
 \qquad \Rightarrow \qquad 
 \norm{g}_{\alpha} = \supnm \sup_{\tau\in\cF_2} \bigabs{c_{\tau,nm}}\;,
\end{equation}
where the supremum over $\tau\in\cF_2$ may be replaced by any other norm on the 
finite-dimensional span of $\cF_2$.}

\added[id=NB]{We summarise the construction in the following definition}.
\added[id=NB]{%
\begin{definition}
Let $\cF_1$ and $\cF_2$ be defined by~\eqref{eq:tauset_KQ} and let 
\begin{equation}
 \cF_3 = \bigsetsuch{\cI^Q\tau}{\tau\in\cF_1} \cup 
 \bigsetsuch{\cI^Q_{nm}\tau}{\tau\in\cF_2,(n,m)\in\setnm}\;.
\end{equation} 
Take as model space the Banach space 
\begin{equation}
 \vspan(\cF_1 \cup \cF_2 \cup \cF_3) \cup \Tbar\;,
\end{equation} 
where $\Tbar$ is the span of all polynomials. If $f\in\cD^{\gamma,\eta}$ is of 
the form~\eqref{eq:f123}, we set 
\begin{equation}
 \cK^Q_\gamma f(z)
 = \cK^Q_\gamma f_1(z) + \cK^Q_\gamma \varphi\unit(z) + \cK^Q_\gamma f_2(z)\;,
\end{equation} 
where $\cK^Q_\gamma f_1$ and $\cK^Q_\gamma \varphi\unit$ are defined 
in~\eqref{eq:def_KQf1} and~\eqref{eq:def_KQf1phi} and $\cK^Q_\gamma f_2$ is 
given in~\eqref{eq:def_KQf2}. 
\end{definition}
\begin{remark}
We could also have introduced symbols of the form $\cI^Q_{nm}\tau$ for 
$\tau\in\cF_1$, but this is not necessary because $f_1(z)$ does not depend on 
$z$. 
\end{remark}
}

In this setting, we can now state our central result, which is the following
extension of the multilevel Schauder estimates in~\cite[Thm.~5.12]{Hairer2014}.
Here the notations for $\normDgamma{f}_{\gamma,\eta;T}$, $\seminormff{f}{\bar
f}_{\gamma,\eta;T}$ and $\seminormff{Z}{\bar Z}_{\gamma;O}$ are as
in~\cite[Def.~6.2]{Hairer2014} and~\cite[Sec.~7.1]{Hairer2014} with $P$ the
hyperplane $\set{t=0}$, see also~\cite[Sec.~4.3]{BK2016}. 

\begin{theorem}
\label{thm:Schauder} 
Let $\alpha_0 = \abss{\RSW}$ be the regularity of the sector defining $f$. 
Assume $f\in\cD^{\gamma,\eta}$ is of the form~\eqref{eq:f123}, where $\eta < 
\alpha_0\wedge\gamma$, and $\gamma+\beta, \eta+\beta \not\in\N$. Then 
$\Rplus\cK^Q_\gamma\Rplus f \in \cD^{\gamma+\beta,\eta+\beta}$ and 
\begin{equation}
\label{eq:schauder1} 
 (\cR\cK^Q_\gamma \Rplus f)(z) = (K^Q * \cR \Rplus f)(z)
\end{equation}
holds for every $z=(t,x)$ such that $t>0$. 
Furthermore, we have 
\begin{equation}
\label{eq:schauder2} 
 \normDgamma{\Rplus\cK^Q_\gamma\Rplus f}_{\gamma+\beta,\bar\eta;T}
 \lesssim T^{\kappa/\fraks_0} \normDgamma{f}_{\gamma,\eta;T}
\end{equation} 
whenever $\bar\eta = \eta+\beta-\kappa$ with $\kappa>0$. 
Finally, if $\bar Z=(\bar\Pi,\bar\Gamma)$ is a second model satisfying 
$\bar\Pi_{z+h_{nm}}\tau = \bar\Pi_z\tau$ for all $\tau\in\cF_1\cup\cF_2$ and 
all $(n,m)\in\setnm$, and $\bar f\in\cD^{\gamma,\eta}(\bar\Gamma)$ is of the 
form~\eqref{eq:f123}, then 
\begin{equation}
\label{eq:schauder3} 
 \seminormff{\Rplus\cK^Q_\gamma\Rplus f}{\Rplus\bar\cK^Q_\gamma\Rplus
\bar f}_{\gamma+\beta,\bar\eta;T}
 \lesssim T^{\kappa/\fraks_0}
\bigpar{\seminormff{f}{\bar f}_{\gamma,\eta;T} + \seminormff{Z}{\bar
Z}_{\gamma;O}}\;.
\end{equation}
\end{theorem}

The proof is given in Appendix~\ref{app:Schauder}. Note that we have assumed
$\eta<\alpha_0$ to simplify the notation (otherwise we need to 
take $\bar\eta = (\eta\wedge\alpha_0)+\beta-\kappa$). 
Note also the extra factor $\Rplus(t,x) = \indexfct{t>0}$, which is needed 
because the translation operators shift singularities along the time axis. 


\section{Fixed point argument}
\label{sec:fixedpoint} 

Assume the nonlinearity has the general cubic form~\eqref{eq:FUV}.
Note in particular that if $p(z)$ and $q(z)$ are polynomial terms, and 
$\Phi(z)$ and $\Psi(z)$ are terms of fractional, strictly positive 
homogeneity, then 
\begin{equation}
 F(\RSI + p(z) + \Phi(z), \RSoI + q(z) + \Psi(z)) 
 = f_1(z) + f_2(z) + F(p(z), q(z)) + f_3(z)\;,
\label{eq:FUV1} 
\end{equation} 
where 
\begin{align}
\nonumber
f_1(z) &= \gamma_1 \RSW + \gamma_2\RSWo + \gamma_3\RSWoo + \gamma_4\RSWooo\;,
\\
f_2(z) &= b_1(z) \RSV + b_2(z) \RSVo + b_3(z) \RSVoo 
+ a_1(z) \RSI + a_2(z) \RSoI\;, 
\end{align}
with 
\begin{align}
\nonumber
b_1(z) &= \beta_1 + 3\gamma_1p(z) +  \gamma_2q(z)\;, \\
\nonumber
b_2(z) &= \beta_2 + 2\gamma_2p(z) + 2\gamma_3q(z)\;, \\
\nonumber
b_3(z) &= \beta_3 +  \gamma_3p(z) + 3\gamma_4q(z)\;. \\
\nonumber
a_1(z) &= \alpha_1 + 2\beta_1p(z) + \beta_2q(z) +
3\gamma_1p(z)^2 +2\gamma_2p(z)q(z) + \gamma_3q(z)^2\;, \\
a_2(z) &= \alpha_2 + \beta_2p(z) + 2\beta_3q(z) +
\gamma_2p(z)^2 +2\gamma_3p(z)q(z) + 3\gamma_4q(z)^2\;.
\label{eq:bc} 
\end{align}
Furthermore, all terms of $f_3(z)$ contain at least a factor $\Phi(z)$ or a 
factor $\Psi(z)$. Thus if the model $\Pi$ satisfies the two properties 
\begin{equation}
\label{eq:model_canonical} 
 \abs{\Pi_z\tau(\bar z)} \lesssim \norm{z-\bar 
z}_\fraks^{\abss{\tau}}\norm{\tau}\;, 
\qquad
\Pi_z(\tau_1\tau_2) = \Pi_z(\tau_1) \Pi_z(\tau_2)
\end{equation} 
for all $\tau,\tau_1,\tau_2\in T$, then $f_3$ satisfies the diagonal 
identity~\eqref{eq:diag_ident}. 

Let $\cD^{\gamma,\eta}_*(\Pi)$ denote the subspace of modelled 
distributions in $\cD^{\gamma,\eta}(\Pi)$ whose components of negative 
homogeneity are of the form $c_1\RSI+c_2\RSoI$ for constants $c_1,c_2\in\R$. 
Consider the map $\cM(U,V) = (\cM_1(U,V),\cM_2(U,V))$ defined on 
$\cD^{\gamma,\eta}_*(\Pi)\times\cD^{\gamma,\eta}_*(\Pi)$ by 
\begin{align}
\nonumber
\cM_1(U,V) &= \Rplus(\cK_{\bar\gamma}+R_\gamma\cR)\Rplus F(U,V) +
W_1\;, \\
\cM_2(U,V) &= \Rplus(\cK^Q_{\bar\gamma}+R^Q_\gamma\cR)\Rplus
F(U,V) + W_2\;,
\label{eq:fix_12} 
\end{align}
where $W_1$ and $W_2$ are placeholders for the stochastic convolution and the
initial conditions (we only need the case where $W_1 - \RSI$ and $W_2 -
\RSoI$ take values in the polynomial part of the regularity structure).
By iterating the map~\eqref{eq:fix_12}, we find that if it admits a fixed point,
then it necessarily has the form 
\begin{alignat}{2}
\nonumber
U(z) ={}& \RSI &&+ \varphi(z) \unit 
+ \bigbrak{\gamma_1\!\RSIW + \gamma_2\!\RSIWo + \gamma_3\!\RSIWoo +
\gamma_4\!\RSIWooo}
+ \bigbrak{b_1(z)\RSY + b_2(z)\RSYo + b_3(z)\RSYoo} + \dots \\
\label{eq:UV_expansion} 
V(z) ={}& \RSoI &&+ \psi(z) \unit 
+ \bigbrak{\gamma_1\!\RSoIW + \gamma_2\!\RSoIWo + \gamma_3\!\RSoIWoo +
\gamma_4\!\RSoIWooo} \\
\nonumber
&&&{}+ \sumnm \bigbrak{b_1(z+h_{nm})\RSoY_{nm} + b_2(z+h_{nm})\RSoYo_{nm} +
b_3(z+h_{nm})\RSoYoo_{nm}} + \dots
\end{alignat}
where the $b_i(z)$ are as in~\eqref{eq:bc} with $p(z)=\varphi(z)$,
$q(z)=\psi(z)$, and the dots indicate terms of homogeneity at least $1$. 
As in~\cite[Prop.~5.2]{BK2016}, it is rather straightforward to show that if
$(U,V)$ satisfies the fixed-point equation~\eqref{eq:fix01} with $U$ and $V$ in
some $\cD^{\gamma,\eta}$ then $(u,v)=(\cR U,\cR V)$ satisfies~\eqref{eq:fhn02}. 

\cite[Prop.~5.6]{BK2016} is then replaced by the following result, which is all 
we need for the fixed-point argument to work. Its proof is very similar to the 
proof of~\cite[Prop.~5.6]{BK2016}, so we omit it here. 

\begin{prop}
Let $\Pi$ be a model satisfying~\eqref{eq:model_canonical}, and assume $-\frac23 
< \eta < \alpha < -\frac12$, $\eta+2\alpha>-2$ and $\gamma>-2\alpha$. Then for 
any $W_1,W_2\in\cD^{\gamma,\eta}_*(\Pi)$ of regularity $\alpha$, there exists a 
time $T>0$ such that $\cM$ admits a unique fixed point $(U^*,V^*) \in 
\cD^{\gamma,\eta}_*(\Pi)\times\cD^{\gamma,\eta}_*(\Pi)$ on $(0,T)$. Furthermore, 
the solution map $\cS_T: (W_1,W_2,Z) \mapsto (U^*,V^*)$ is jointly Lipschitz 
continuous.  
\end{prop}

Note that in the case  
\begin{align}
\nonumber
W_1 &= (\cK_{\bar\gamma}+R_\gamma\cR)\Rplus \Xi + Gu_0\;, \\
W_2 &= (\cK^Q_{\bar\gamma}+R_\gamma\cR)\Rplus \Xi + G^Q u_0 + \widehat Q v_0\;,
\end{align}
the fixed point $(U^*,V^*)$ of $\cM$ is indeed a fixed point 
of~\eqref{eq:fix01}. As pointed out in~\cite[Rem.~5.7]{BK2016}, the assumptions 
on $u_0$ and $v_0$ guarantee that $W_1$ and $W_2$ belong to the right 
functional space.


\section{Renormalisation}
\label{sec:renormalisation} 

It remains to check that the fact that we have modified the regularity 
structure by adding a countable infinity of symbols does not cause any problems 
as far as the renormalisation procedure is concerned, and to derive the 
renormalised equations. 

We define a renormalisation transformation, depending on two parameters, given 
by 
\begin{equation}
 M_\eps\tau = \exp \Bigset{-C_1(\eps) L_1\tau - C_2(\eps) L_2\tau}\;, 
\end{equation} 
where the generators $L_1$ and $L_2$ are defined by applying the substitution 
rules (called contractions)
\begin{equation}
 L_1 : \RSV \mapsto \unit\;, 
 \qquad 
 L_2 : \RSWV \mapsto \unit
\end{equation} 
as many times as possible, so that for instance $L_1 \RSW = 3 \RSI$. In 
particular, we obtain 
\begin{equation}
 M_\eps \RSoWV_{nm} = \RSoWV_{nm} - C_1(\eps) \RSoY_{nm}\;.
\end{equation} 
Other examples of the action of $M_\eps$ are given in~\cite[(6.12)]{BK2016}. 
Note that there are no generators acting by contracting symbols that 
contain at least one edge associated to $\cI^Q$, implying that for instance 
$M_\eps\RSVo = \RSVo$ and $M_\eps\RSoWoVo_{nm} = \RSoWoVo_{nm}$. 
The fact that these symbols do not require additional 
renormalisation constants is a consequence of~\cite[Lem.~6.2]{BK2016} and 
Lemma~\ref{lem:I00} below.

The renormalisation map $M_\eps$ induces a renormalised model 
$\widehat\Pi^\eps=\Pi^{M_\eps}$ which can be computed as described 
in~\cite[Sect.~8.3]{Hairer2014} 
and~\cite[Sect.~6.1]{BK2016}. In particular, we find 
\begin{equation}
 \widehat\Pi^\eps_z(\RSoWV_{nm}) 
 = \Pi^\eps_z(\RSoY_{nm}) \widehat\Pi^\eps_z(\RSV)\;. 
\end{equation} 
Here the canonical model for $\Pi^\eps_z(\RSoY_{nm})$ can be computed 
using~\eqref{eq:model_IQnm}, which yields  
\begin{align}
\nonumber
(\Pi^\eps_z \RSoY_{nm})(\bar z) 
&= \hat\chi^0_{\smallRSV,nm}(\bar z) - \hat\chi^0_{\smallRSV,nm}(z) \\
\nonumber
&= \pscal{\Pi^{\eps,nm}_z\RSV}{\KQhat_{nm}(\bar z-\cdot) - 
\KQhat_{nm}(z-\cdot)} \\
\nonumber
&= \pscal{\Pi^{\eps}_z\RSV}{K^Q_{nm}(\bar z-\cdot) - 
K^Q_{nm}(z-\cdot)} \\
&= \int \bigbrak{K^Q_{nm}(\bar z-z_1) - K^Q_{nm}(z-z_1)} 
\bigbrak{(K_\eps * \xi)(z_1)}^2 \6z_1\;,
\end{align}
where $K_\eps = K * \varrho_\eps$, and we have used the 
expression for the canonical model of $\RSV$ in the last line, 
which is base-\added[id=NB]{time} independent, cf.~\cite[(6.28)]{BK2016}. It 
follows that 
\begin{align}
\nonumber
  \widehat\Pi^\eps_z(\RSoWV_{nm})(\bar z)
  = \int & \bigbrak{K^Q_{nm}(\bar z-z_1) - K^Q_{nm}(z-z_1)} 
\bigbrak{(K_\eps * \xi)(z_1)}^2 \6z_1 \\
& \times \Bigpar{\bigbrak{(K_\eps * \xi)(\bar z)}^2 - C_1(\eps)}\;.
\end{align} 
The renormalised models of other symbols are obtained in a similar way, using 
the expressions given in~\cite[(6.13)]{BK2016}. 

We now have to show that the renormalised models converge, for an appropriate 
choice of the renormalisation constants $C_1(\eps)$ and $C_2(\eps)$, to a 
well-defined limiting model. This amounts to showing that the Wiener chaos 
expansions of the renormalised models satisfy the bounds~\cite[(6.20)]{BK2016}. 
To a large extent, the computations have already been made 
in~\cite[Prop.~6.4]{BK2016}, so that we only discuss one representative case 
involving an infinite collection of symbols. Proceeding as 
in~\cite[(6.39)]{BK2016}, we find that the contribution to the zeroth Wiener 
chaos of $\widehat\Pi^\eps_z(\RSoWV_{nm})$ is given by 
\begin{align}
 (\widehat\cW^{(\eps;0)}_z \RSoWV_{nm})(\bar z)
 = 2\iiint &\bigbrak{K^Q_{nm}(\bar z-z_1) - K^Q_{nm}(z-z_1)} \\
 &\times K_\eps(z_1-z_2) K_\eps(z_1-z_3) K_\eps(\bar z-z_2) K_\eps(\bar z-z_3) 
 \6z_1 \6z_2 \6z_3\;.
 \nonumber
\end{align} 
As in~\cite[Prop.~6.4]{BK2016}, the crucial term is the one involving 
$K^Q_{nm}(\bar z-z_1)$, which can be rewritten as $2I^Q_{00;nm}(\eps)$, 
where  
\begin{equation}
 I^Q_{00;nm}(\eps) = \int K^Q_{nm}(z_1) \cQ^\eps_0(z_1)^2 \6z_1\;, 
 \qquad 
 \cQ^\eps_0(z) = \int K_\eps(z_1) K_\eps(z_1-z)\6z_1\;.
\end{equation} 
Note that if $K^Q_{nm}$ is replaced by $K$, we obtain the 
renormalisation constant $C_2(\eps)$, which diverges like $\log(\eps^{-1})$, 
cf.~\eqref{eq:C1_C2}. The following lemma implies that no renormalisation is 
needed in the case of $\RobustRSoWVnm$. 

\begin{lemma}
\label{lem:I00} 
For all $(n,m) \in \setnm$, the bound
\begin{equation}
\label{eq:IQ00_nm} 
 \bigabs{I^Q_{00;nm}(\eps)} 
 \lesssim \frac{2^{-2n}}{m+2}
\end{equation}
holds uniformly in $\eps\in[0,1]$. 
\end{lemma}

The proof is given in Appendix~\ref{app:renorm1}. The important point is that 
the bound~\eqref{eq:IQ00_nm} is square-summable over all $(n,m)\in\setnm$, which 
is related to the fact that $K^Q(z_1) \cQ^\eps_0(z_1)^2$ is integrable uniformly 
in $\eps$. This is essential in establishing the following convergence result.

\begin{prop}
Let $C_1(\eps)$ and $C_2(\eps)$ be the constants defined in~\eqref{eq:C1_C2}. 
Then there exists a random model $\widehat Z = (\widehat \Pi,\widehat \Gamma)$, 
independent of the choice of mollifier $\varrho$, such that for any $\theta < 
-\frac52 - \alpha_0 = \kappa$ and any compact set $\fraK$, one has 
\begin{equation}
 \E \seminormff{\widehat Z^\eps}{\widehat Z}_{\gamma;\fraK} \lesssim 
\eps^\theta\;,
\end{equation} 
provided $\gamma < \zeta$, where $\zeta$ is such that all moments of $K$ up to 
parabolic degree $\zeta$ vanish. 
\end{prop}
\begin{proof}
The proof follows along the lines of~\cite[Prop.~6.4]{BK2016}, which is closely 
based on~\cite[Thm.~10.22]{Hairer2014}. The crucial point to note here is 
that~\cite[Thm.~10.7]{Hairer2014} can still be applied in this case, even though 
there is a countable infinity of symbols such as $\RSoWV_{nm}$ that need to be 
renormalised. Indeed, the bound~\cite[(10.4)]{Hairer2014} involves a sum, over 
all basis vectors $\tau$ of a given sector, of the $p$-th power of the second 
moment of $\pscal{\widehat\Pi_0\tau}{\psi^\lambda_0}$. To be applicable, the 
bounds
\begin{align}
\nonumber
{\pscal{\widehat \Pi_z\RSoWV_{nm}}{\psi^\lambda_z}}^2 
&
\leqs C^2_{nm} \lambda^{2\abs{\tau}_\fraks+\kappa}\;, \\
\E\bigabs{\pscal{\widehat \Pi_z\RSoWV_{nm} - \widehat
\Pi_z^{\eps}\RSoWV_{nm}}{\psi^\lambda_z}}^2 
&
\leqs C^2_{nm} \eps^{2\theta}\lambda^{2\abs{\tau}_\fraks+\kappa}
\end{align}
should hold for some $\kappa, \theta > 0$, with proportionality 
constants $C^2_{nm}$ that are summable over all $(n,m)\in\setnm$, 
cf.~\cite[(10.2), (10.3)]{Hairer2014}.
By~\cite[Prop.~10.11]{Hairer2014}, this is the case if 
the Wiener chaos expansion of these $\tau$ satisfies the 
bounds
\begin{align}
\nonumber
\Bigabs{\bigpscal{(\widehat\cW^{(k)}\RSoWV_{nm})(z)}{(\widehat\cW^{(k)}
\RSoWV_{nm})(\bar z)}}
&\leqs C_{nm}^2 \sum_{\varsigma>0} \bigpar{\norm{z}_{\fraks} +
\norm{\bar z}_{\fraks}}^\varsigma \norm{z-\bar
z}_{\fraks}^{\bar\kappa+2\alpha-\varsigma}\;, \\
\Bigabs{\bigpscal{(\delta\widehat\cW^{(\eps;k)}\RSoWV_{nm})(z)}
{(\delta\widehat\cW^{(\eps;k)} \RSoWV_{nm})(\bar z)}}
&
\leqs C_{nm}^2 \eps^{2\theta} \sum_{\varsigma>0} 
\bigpar{\norm{z}_{\fraks} +
\norm{\bar z}_{\fraks}}^\varsigma \norm{z-\bar
z}_{\fraks}^{\bar\kappa+2\alpha-\varsigma}
\label{eq:rFHN04c}
\end{align}
for some $\bar\kappa,\theta>0$, where $\alpha=\abss{\RSoWV_{nm}}$ and the sums 
run over finitely many positive $\varsigma$. This in turns follows from the 
square-summability of integrals such as~\eqref{eq:IQ00_nm}. 
\end{proof}

The final step is to compute the renormalised equations corresponding to the 
renormalisation map $M_\eps$. It is straightforward to check that Lemma~6.5 and 
Proposition~6.7 in~\cite{BK2016} still hold in the present situation. It is 
thus sufficient to compute the non-positive-homogeneous part $\widehat F(U,V)$ 
of $M_\eps F(U,V)$, for a cubic nonlinearity $F$ as in~\eqref{eq:FUV}. This 
yields the following result, which is proved in 
Appendix~\ref{app:renorm2}. 

\begin{prop}
\label{prop:Fhat} 
In the situation just described, we have 
\begin{equation}
 \widehat F(u,v) 
 = F(u,v) + c_0(\eps) + c_1(\eps)u + c_2(\eps)v\;,
\end{equation} 
where the $c_i(\eps)$ are defined in~\eqref{eq:c0c1c2}. 
\end{prop}

The proof of Theorem~\ref{thm:main} now follows in the same way as 
in~\cite[Sec.~7]{BK2016}. 


\paragraph*{Acknowledgements.}

We would like to thank Tom Holding and Martin Hairer for pointing out the error 
in~\cite[Lem.~4.8]{BK2016}, and Tom, Martin, Yvain Bruned, Cyril Labb\'e and 
Hendrik Weber for their advice on preliminary versions of this erratum. 
\added[id=NB]{We also thank the anonymous referee for providing constructive 
comments, that have led to improvements in the presentation.}


\appendix

\section{Proof of Proposition~\ref{prop:KQnm}}
\label{app:KQnm} 

Let $\varphi:\R\to[0,1]$ be a partition of unity, i.e., a
function satisfying
\begin{itemize}
\item	$\varphi$ is of class $\cC^\infty$ and of compact support, say $[-1,1]$;
\item	for any $t\in\R$ one has 
\begin{equation}
 \sum_{m\in\Z} \varphi(m+t) = 1\;.
\end{equation} 
\end{itemize}

\begin{remark}
An example of such a function would be a smooth even function, satisfying
$\varphi(\theta) = 1 - \varphi(1-\theta)$ for any $\theta\in[0,1]$. For instance 
one can take
\begin{equation}
 \varphi(\theta) = \frac12 + \frac12\tanh\biggpar{\frac{1}{\theta} - 
\frac1{1-\theta}}
 \qquad \forall \theta\in(0,1)
\end{equation} 
and set $\varphi(0)=1$, $\varphi(\theta)=0$ for $\theta\geqs1$ and
$\varphi(-\theta)=\varphi(\theta)$ for all $\theta$. 
\end{remark}

Given such a function $\varphi$, we set 
\begin{equation}
 \varphi_n(\theta) = \varphi(2^{\fraks_0n} \theta) 
\end{equation} 
for all $n\in\N_0$. Then $\varphi_n$ is supported on
$[-2^{-\fraks_0n},2^{-\fraks_0n}]$. Observe that for any $n\in\N_0$ and any
$t\in\R$, we have 
\begin{equation}
 \sum_{m\in\Z} \varphi_n(2^{-\fraks_0 n}m + t) = 1\;,
\end{equation} 
and thus 
\begin{equation}
 \sum_{m\in\Z} Q_{nm}(t) = Q(t)
 \qquad
 \text{where}
 \quad 
 Q_{nm}(t) = Q(t)  \varphi_n(2^{-\fraks_0 n}m + t)\;.
\end{equation} 
Since $Q$ is compactly supported, the above sum only contains a finite number
of terms, of order~$2^{\fraks_0 n}$. In fact, for any given $t$, there are at
most two nonzero terms in the sum, and $Q_{nm}$ is supported on the
interval 
\begin{equation}
\label{eq:support_Qnm} 
 [(m-1)2^{-\fraks_0 n}, (m+1)2^{-\fraks_0 n}]\;.
\end{equation} 
Recall that the kernel $K$ being regularizing of order $\beta$ 
means that $K$ and its derivatives satisfy the bounds given 
in~\cite[Assumption~5.1]{Hairer2014}.
Thus if we define for $n\in\N_0$, $m\in\Z$  
\begin{equation}
\nonumber
 K^Q_{nm}(t,x) = \int_{t-2T}^t Q_{nm}(t-s) K_n(s,x) \6s
 = \int_0^{2T} Q_{nm}(u) K_n(t-u,x) \6u\;,
\end{equation} 
then we obtain a decomposition
\begin{equation}
 K^Q(z) = \sum_{n\geqs0} \sum_{m=-1}^{1+2T2^{\fraks_{0}n}} K^Q_{nm}(z)\;,
\end{equation} 
where the range of $m$ is due to~\eqref{eq:support_Qnm} and the fact that $Q$
is supported on $[0,2T]$. 

 First note that since $Q_{nm}$ is supported on the interval
given in~\eqref{eq:support_Qnm}, $K^Q_{nm}(t,x)$ can be nonzero only if 
\begin{equation}
 (m-2)2^{-\fraks_0 n} \leqs t \leqs (m+2)2^{-\fraks_0 n}
\end{equation} 
and $x$ is in a ball of radius $2^{-n}$. The condition on $t$ is equivalent to
$\abs{t-m2^{-\fraks_0n}}\leqs 2^{1-\fraks_0 n}$, from which~\eqref{eq:KQnm_1}
follows. 

Since $\varphi$ takes values in $[0,1]$, we have 
\begin{equation}
 \abs{Q_{nm}(t)} \leqs \abs{Q(t)} \leqs \norm{Q}_\infty
 := \sup_{t\in[0,2T]} \abs{Q(t)}\;.
\end{equation} 
Therefore, by Condition~(5.4) in~\cite[Assumption~5.1]{Hairer2014}, we have
\begin{align}
\nonumber
\abs{\D^k K^Q_{nm}(z)}
&\leqs \int_{(m-1)2^{-\fraks_0n}}^{(m+1)2^{-\fraks_0n}}
\abs{Q_{nm}(u)} \abs{\D^k K_n(t-u,x)} \6u \\
\nonumber
&\leqs C 2^{(\abs{\fraks}-\beta+\abss{k})n}
\int_{(m-1)2^{-\fraks_0n}}^{(m+1)2^{-\fraks_0n}} \abs{Q_{nm}(u)}\6u \\
&\leqs 2C 2^{(\abs{\fraks}-\fraks_0-\beta+\abss{k})n} \norm{Q_{nm}}_\infty\;,
\end{align}
which implies~\eqref{eq:KQnm_2} with $C_Q = 2C\norm{Q}_\infty$. 

Finally, we have 
\begin{align}
\nonumber 
\int_{\R^{d+1}}z^\ell \D^k K^Q_{nm}(z)\6z 
&= \int_{(m-1)2^{-\fraks_0n}}^{(m+1)2^{-\fraks_0n}} Q_{nm}(u) 
\int_{\R^{d+1}} z^\ell \D^k K_n(t-u,x)\6z \6u \\
&= \int_{(m-1)2^{-\fraks_0n}}^{(m+1)2^{-\fraks_0n}} Q_{nm}(u) 
\int_{\R^{d+1}} (\bar z + (u,0))^\ell \D^k K_n(\bar z)\6\bar z \6u\;.
\end{align}
It follows from Condition~(5.5) in~\cite[Assumption~5.1]{Hairer2014}, applied
to all $\ell'$ of degree less or equal $\ell$, that there exists a constant
$C'$, depending only on $k$ and $\ell$, such that the absolute value of the
integral over $\R^{d+1}$ is bounded by $C'2^{-\beta n}$ uniformly in $n$.
Therefore, \eqref{eq:KQnm_3} follows with $C_Q=2C'$.
\qed


\section{Proof of Lemma~\ref{lem:modelIQ}}
\label{app:model} 

As in~\cite[Lem.~5.19]{Hairer2014}, the cases $2^{-n} > \lambda$ and
$2^{-n}\leqs\lambda$ are treated differently. We start by dealing with the case
$2^{-n} > \lambda$. 
The \added[id=NB]{base-time independence} assumption $\Pi_{z+h_{nm}}\tau = 
\Pi_z\tau$ implies
\begin{equation}
 \bigpscal{\Pi\added[id=NB]{_z}\tau}{\D^k K^Q_{nm}(z-\cdot)}
 =  \bigpscal{\Pi_{z+h_{nm}}\tau}{\D^k K^Q_{nm}(z-\cdot)}\;.
\end{equation} 
Since the singularity of $K^Q_{nm}(z-\cdot)$ is located at 
$z+h_{nm}$, we can apply~\cite[Rem.~2.21]{Hairer2014}, which together with the
bound~\eqref{eq:KQnm_2} on $\abs{\D^k K^Q_{nm}}$ yields 
\begin{equation}
\label{eq:bound_Pitau_DkKQ} 
 \bigabs{\bigpscal{\Pi\added[id=NB]{_z}\tau}{\D^k K^Q_{nm}(z-\cdot)}}
 \lesssim \norm{\Pi}_{\alpha;\fraK_z} 2^{(\abss{k}-\fraks_0-\alpha-\beta)n}\;.
\end{equation} 
Note that owing to base-\added[id=NB]{time} independence of the model, we have 
avoided making use of $\Gamma$ as in~\cite[Lem.~5.18]{Hairer2014}. We now use 
the Taylor expansion representation of~\cite[Appendix~A]{Hairer2014} to get 
\begin{equation}
\label{eq:KQalpha_int} 
  K^{Q;\alpha}_{nm;z\bar z}(z') 
 = \sum_{\ell\in\partial A} \int_{\R^{d+1}} \D^\ell K^Q_{nm}(\bar
z+h-z')\cQ^\ell(z-\bar z,\6h)
\end{equation} 
where $A = \setsuch{\ell}{\abss{\ell} < \alpha+\beta}$ and $\cQ^\ell$ is a
measure with total mass $\norm{z-\bar z}_\fraks^{\abss{\ell}}$. 
It follows from~\eqref{eq:bound_Pitau_DkKQ} that 
\begin{equation}
 \bigabs{\bigpscal{\Pi\added[id=NB]{_z}\tau}{K^{Q;\alpha}_{nm;z\bar z}}} 
 \lesssim \norm{\Pi}_{\alpha;\fraK_z} 
 \sum_{\ell\in\partial A} \norm{z-\bar z}_\fraks^{\abss{\ell}} 
 2^{(\abss{\ell}-\fraks_0-\alpha-\beta)n}\;.
\end{equation} 
Together with the fact that 
\begin{equation}
\label{eq:bound_integral_power} 
 \int_{\R^{d+1}} \norm{z-\bar z}^{\abs{\ell}_\fraks}_\fraks
\psi^\lambda_z(\bar z)\6\bar z \lesssim \lambda^{\abs{\ell}_\fraks}
\end{equation}
this yields 
\begin{align}
\nonumber
 \sumnmarg{2^{-n} > \lambda}  
 \biggabs{\int_{\R^{d+1}}  
\bigpscal{\Pi\added[id=NB]{_z}\tau}{K^{Q;\alpha}_{nm;z\bar
z}} \psi^\lambda_z(\bar z) \6\bar z}
& \lesssim \sum_{\ell\in\partial A} \sumnmarg{2^{-n} > \lambda} 
2^{(\abss{\ell}-\fraks_0-\alpha-\beta)n} \lambda^{\abs{\ell}_\fraks}
\norm{\Pi}_{\alpha;\fraK_z} \\
& \lesssim \lambda^{\alpha+\beta} \norm{\Pi}_{\alpha;\fraK_z}\;.
\end{align} 
In the case $2^{-n}\leqs \lambda$, we use the representation 
\begin{equation}
\label{eq:decomp_YZ} 
 \int_{\R^{d+1}} \bigpscal{\Pi\added[id=NB]{_z}\tau}{K^{Q;\alpha}_{nm;z\bar z}}
\psi^\lambda_z(\bar z) \6\bar z 
 = \bigpscal{\Pi\added[id=NB]{_z}\tau}{Y^\lambda_{nm}} 
 - \sum_{\abss{\ell} < \alpha+\beta} 
\bigpscal{\Pi\added[id=NB]{_z}\tau}{Z^{\lambda}_{nm;\ell}}
\end{equation} 
where 
\begin{align}
\nonumber
Y^\lambda_{nm}(z') &= 
 \int_{\R^{d+1}} K^Q_{nm} (\bar z - z') \psi^\lambda_z(\bar z)\6\bar z\;, \\
 Z^\lambda_{nm;\ell}(z') &= 
 \D^\ell K^Q_{nm} (\bar z - z') 
 \int_{\R^{d+1}} \frac{(\bar z-z)^\ell}{\ell!}\psi^\lambda_z(\bar z)\6\bar z\;.
\end{align}
Here one readily checks that the arguments used in the proof
of~\cite[Lem.~5.19]{Hairer2014} to bound 
$\pscal{\Pi\added[id=NB]{_z}\tau}{Y^\lambda_{nm}}$ and
$\pscal{\Pi\added[id=NB]{_z}\tau}{Z^\lambda_{nm;\ell}}$ are not affected by the 
location of the
support of $K^Q_{nm}$, so that as a result we obtain in the same way as there 
the bounds 
\begin{align}
\nonumber
 \bigabs{\bigpscal{\Pi\added[id=NB]{_z}\tau}{Y^\lambda_{nm}}} 
 &\lesssim \lambda^\alpha 2^{-(\fraks_0+\beta)n}\;, \\
 \bigabs{\bigpscal{\Pi\added[id=NB]{_z}\tau}{Z^\lambda_{nm;\ell}}}
 &\lesssim \lambda^{\abss{\ell}} 2^{(\abss{\ell}-\fraks_0-\alpha-\beta)n}\;.
 \label{eq:bounds_YZ} 
\end{align} 
This yields 
\begin{equation}
 \biggabs{\int_{\R^{d+1}} 
\bigpscal{\Pi\added[id=NB]{_z}\tau}{K^{Q;\alpha}_{nm;z\bar z}}
\psi^\lambda_z(\bar z) \6\bar z} 
\lesssim 2^{-\fraks_0 n} \lambda^{\alpha+\beta} 
\sum_{\abss{\ell} < \alpha+\beta} (\lambda 2^n)^{\abss{\ell}-\alpha-\beta}\;,
\end{equation} 
and summing over $(n,m)$ with $2^{-n}\leqs \lambda$ gives the result. 
\qed


\section{Proof of Lemma~\ref{lem:Gammazz}}
\label{app:Gammazz} 

Using the fact that $\pscal{f_z}{X^\ell} = (-z)^\ell$ and multiplicativity of
$\pscal{f_z}{\cdot}$, we obtain 
\begin{equation}
 \bigpscal{f_z}{X^\ell\cJ^Q_{k+\ell}\tau} = -(-z)^\ell
\chi^{k+\ell}_\tau(z)\;.
\end{equation} 
From the expression~\eqref{eq:DeltaIQ} of $\Delta(\cI^Q\tau) $ we thus deduce 
\begin{equation}
\label{eq:FzIQ} 
F_z \cI^Q\tau 
= (\Id\otimes f_z)\Delta(\cI^Q\tau) 
= \cI^Q\tau - \sum_{\abss{k+\ell}<\alpha+\beta} \frac{X^k}{k!} 
\frac{(-z)^\ell}{\ell!} \chi_\tau^{k+\ell}(z)\;.
\end{equation}
In the basis $(\set{X^k}_{\abss{k}<\alpha+\beta},\cI^Q\tau)$ we can thus
identify $F_z$ and its inverse with matrices 
\begin{equation}
 F_z = 
 \begin{pmatrix}
 T(z) & T_*(z) \\ 0 & 1  
 \end{pmatrix}\;, 
 \qquad
 F_z^{-1} = 
 \begin{pmatrix}
 T(z)^{-1} & -T(z)^{-1}T_*(z) \\ 0 & 1  
 \end{pmatrix}\;.
\end{equation} 
Here $T(z)$, which represents the action of $F_z$ on monomials $X^k$, is an
upper triangular matrix with elements 
\begin{equation}
 \bigbrak{T(z)}_{kj} = \frac{j!}{k!(j-k)!}(-z)^{j-k}\;,
\end{equation}
while $T_*(z)$ is a column vector given by the coefficients of $X^k$ in the sum
on the right-hand side of~\eqref{eq:FzIQ}. 
It follows that $\Gamma_{z\bar z}$ is represented by the matrix 
\begin{equation}
 F_z^{-1} F_{\bar z} =
 \begin{pmatrix}
 T(z)^{-1}T(\bar z) & T(z)^{-1}[T_*(\bar z) - T_*(z)] \\ 
 0 & 1
 \end{pmatrix}\;.
\end{equation} 
Since $F_z^{-1} = F_{-z}$ for elements of the polynomial part 
of the regularity structure, one has $T(z)^{-1}=T(-z)$. A direct computation of
the upper right matrix element then yields the result, making use of the
binomial identity. 
\qed

\begin{remark}
Another way of deriving the result is by using the identity 
\begin{equation}
 F_z^{-1} \cI^Q\tau = \cI^Q\tau 
 + \sum_{\abss{k} < \alpha+\beta} \frac{X^k}{k!} \chi^k_\tau(z)\;,
\end{equation} 
which can be readily checked by showing that $F_z^{-1}F_z = \Id$. 
\end{remark}


\section{Proof of Theorem~\ref{thm:Schauder}}
\label{app:Schauder} 

It follows from~\cite[Prop.~6.16 and Thm.~7.1]{Hairer2014} that $\cK^Q_\gamma
f_1$ and $\cK^Q_\gamma \varphi\unit$ satisfy the theorem. It thus remains to
prove the statement for $\cK^Q_\gamma f_2$ and check the convolution identity
for $f_3$. By linearity, it is sufficient to consider the case 
\begin{equation}
\label{eq:def_f2} 
 f_2(z) = a(z)\tau + \sum_{i=1}^3 a_i(z) X_i\tau\;, 
 \qquad 
 \tau\in\set{\RSV,\RSVo,\RSVoo}\;,
\end{equation} 
the cases $f_2(z)=a(z)\tau$ with $\tau\in\set{\RSI,\RSoI}$ being similar but 
simpler. Therefore we fix $\alpha=\abss{\tau} =-1-2\kappa$. To further simplify 
the notations, we will drop the notation $\sum_{i=1}^3$, and not indicate the 
dependence of proportionality constants on $\norm{\Pi}$ and $\norm{\Gamma}$.

It is crucial to keep track of the sign of the first component $t$ of $z$.  
We write $\Rplus f_2(z) = a^+(z)\tau + a^+_i(z)X_i\tau$, where $a^+(z) = 
a(z)\indexfct{t>0}$ and similarly for $a^+_i(z)$, and
$t_+ = t\indexfct{t>0}$. We also use the notations $f^+_{nm}(z) =
\ensuremath{f_2}(z+h_{nm})\indexfct{t>0}$ and  
\begin{equation}
 b^{k,+}_{nm}(z) = \pscal{\cR^{nm}f^+_{nm} -
\Pi^{nm}_z f^+_{nm}(z)}{\D^k\KQhat_{nm}(z-\cdot)}\;. 
\end{equation}
For all $z,\bar z\in\R^{d+1}$, $(n,m)\in\setnm$
and $\tau$ as in~\eqref{eq:def_f2}, we have the relations 
\begin{equation}
 \Pi^{nm}_z\tau = \Pi^{nm}_{\bar z}\tau\;, \qquad 
 \Pi^{nm}_zX_i\tau = \Pi^{nm}_{\bar z}X_i\tau
 + (\bar z_i-z_i)\Pi^{nm}_z\tau\;,
\end{equation} 
and the estimates
\begin{align}
\label{eq:bound_psi} 
 \bigabs{\hat\chi^k_{\tau,nm}(z)}
 &\lesssim 2^{(\abss{k}-\fraks_0-\alpha-\beta)n}\;, \\
\label{eq:bound_Xipsi} 
 \bigabs{\hat\chi^k_{X_i\tau,nm}(z)}
 &\lesssim 2^{(\abss{k}-\fraks_0-\alpha-\beta-1)n}\;, \\
\label{eq:bound_bk} 
\bigabs{\Rplus b^{k,+}_{nm}(z)} &\leqs 
(1\wedge t_+)^{(\eta-\gamma)/\fraks_0} \indexfct{t>0}
2^{(\abss{k}-\fraks_0-\gamma-\beta)n} 
\normDgamma{\Rplus f_2}_{\gamma,\eta;\bar\fraK}\;.
\end{align}
Indeed, the first two bounds follows in the same way 
as~\eqref{eq:bound_Pitau_DkKQ} (using the fact that $\bar z_i-z_i=0$ for 
time-translations), while the last one is a consequence of the 
improved reconstruction theorem~\cite[Lem.~6.7]{Hairer2014}, \eqref{eq:KQnm_2} 
and the fact that $f^+_{nm} \in \cD^{\gamma,\eta}(\Pi^{nm})$. 
This shows that the infinite series in the definition~\eqref{eq:def_KQf2} of 
$\cK^Q_\gamma$ are indeed convergent.
Note that this is exactly the point where the introduction of 
shifted models is necessary, since otherwise $b^k_{nm}$ would only satisfy a  
weaker bound of the form~\eqref{eq:weakbound}, which guarantees summability 
only for $\abss{k} < \alpha+\beta$.

\subsection{Bounds on $\normDgamma{\Rplus\cK^Q_\gamma\Rplus
f}_{\gamma+\beta,\bar\eta;T}$}

Since $\Gamma_{z\bar z}\tau = \tau$ and $\Gamma_{z\bar z}X_i\tau = 
X_i\tau+(z_i-\bar z_i)\tau$, the fact that $f_2\in\cD^{\gamma,\eta}$ implies 
that 
\begin{align}
\label{eq:bound_a1} 
\abs{a^+(z)} 
&\lesssim
(1\wedge t_+)^{(\eta-\alpha)/\fraks_0} \indexfct{t>0}
\normDgamma{\Rplus f_2}_{\gamma,\eta;\bar\fraK}\;, \\
\nonumber
\abs{a^+_i(z)}  
&\lesssim 
 (1\wedge t_+)^{(\eta-\alpha-1)/\fraks_0} \indexfct{t>0}
\normDgamma{\Rplus f_2}_{\gamma,\eta;\bar\fraK}\;, \\
\nonumber
\abs{a^+(z)-a^+(\bar z)-(z_i-\bar z_i)a^+_i(\bar z)}  
&\lesssim
\norm{z-\bar z}_\fraks^{\gamma-\alpha}
(1\wedge t_+\wedge\bar t_+)^{(\eta-\gamma)/\fraks_0}
\indexfct{t,\bar t>0}
\normDgamma{\Rplus f_2}_{\gamma,\eta;\bar\fraK}\;,\\
\nonumber
\abs{a^+_i(z)-a^+_i(\bar z)}  
&\lesssim
\norm{z-\bar z}_\fraks^{\gamma-\alpha-1}
(1\wedge t_+\wedge\bar t_+)^{(\eta-\gamma)/\fraks_0}
\indexfct{t,\bar t>0}
\normDgamma{\Rplus f_2}_{\gamma,\eta;\bar\fraK}
\end{align}
holds for all $z,\bar z\in\bar\fraK$ (the $1$-fattening of $\fraK$). We start by
estimating 
\begin{equation}
 \norm{\Rplus\cK^Q_\gamma\Rplus f_2}_{\gamma+\beta,\bar\eta;\fraK}
 = \sup_{z\in\fraK}
\max_{\delta\in\set{\alpha+\beta,\alpha+\beta+1,0,\dots,\intpart{\gamma+\beta}}}
\frac{\norm{\Rplus\cK^Q_\gamma\Rplus f_2}_\delta}{1\wedge
t_+^{(\bar\eta-\delta)\wedge0}}\;.
\end{equation} 
In the case $\delta=\alpha+\beta$, we have
\begin{align}
\nonumber
\norm{\Rplus\cK^Q_\gamma\Rplus f_2(z)}_{\alpha+\beta} 
&= \supnm \abs{a^+(z+h_{nm})}\indexfct{t>0} \\
&\leqs (1\wedge t_+)^{(\eta-\alpha)/\fraks_0}\indexfct{t>0} \normDgamma{\Rplus
f_2}_{\gamma,\eta;\bar\fraK}\;.
\end{align}
For $\bar\eta = \eta+\beta$, this provides the first bound required to obtain
$\Rplus\cK^Q_\gamma\Rplus f_2 \in \cD^{\gamma+\beta,\eta+\beta}$. 
Note that the factor $\indexfct{t>0}$, which is due to the first $\Rplus$, is
required to kill the singularities of $a^+(z+h_{nm})$ for negative time. 
In the particular case $\fraK = O_T = (-\infty,T]\times\R^d$, we can further
bound the factor $(1\wedge t_+)^{(\eta-\alpha)/\fraks_0}\indexfct{t>0}$ by
$T^{\kappa/\fraks_0}(1\wedge
t_+)^{(\eta-\alpha-\kappa)/\fraks_0}\indexfct{t>0 }$, with $O_T$
instead of $\bar\fraK$ since the kernel is non-anticipative. This yields one
of the bounds required to prove~\eqref{eq:schauder2}. 
The case $\delta=\alpha+\beta+1$ is treated similarly. 

For polynomial components of exponent $\ell\in\N$, 
\eqref{eq:def_KQf2} implies
\begin{align}
\nonumber
 \norm{\Rplus\cK^Q_\gamma\Rplus f_2(z)}_{\ell}
 \lesssim 
 \sum_{\abss{k}=\ell} \frac{1}{k!}\biggl| \sumnm \Bigl[
& b^{k,+}_{nm}(z) \indexfct{\ell<\gamma+\beta}
+ a^+(z+h_{nm})\hat\chi^k_{\tau,nm}(z)\indexfct{\ell<\alpha+\beta} \\
&{}+ a^+_i(z+h_{nm})\hat\chi^k_{X_i\tau,nm}(z)\indexfct{\ell<\alpha+\beta+1}
\Bigr]\biggr|\indexfct{t>0}\;.
\label{eq:RKQRf2l} 
\end{align} 
Here we treat separately the cases $1\wedge
t_+^{1/\fraks_0} \geqs 2^{-n}$ and $1\wedge t_+^{1/\fraks_0} < 2^{-n}$. In the
first case, we use the bounds
\begin{align}
\nonumber
\bigabs{a^+(z+h_{nm})\hat\chi^k_{\tau,nm}(z)} 
&\lesssim 
(1\wedge t_+)^{(\eta-\alpha)/\fraks_0} 
2^{(\abss{k}-\fraks_0-\alpha-\beta)n}
\normDgamma{\Rplus f_2}_{\gamma,\eta;\bar\fraK}\;, \\
\bigabs{a^+_i(z+h_{nm})\hat\chi^k_{X_i\tau,nm}(z)} 
&\lesssim 
(1\wedge t_+)^{(\eta-\alpha-1)/\fraks_0} 
2^{(\abss{k}-\fraks_0-\alpha-\beta-1)n}
\normDgamma{\Rplus f_2}_{\gamma,\eta;\bar\fraK}\;, 
\label{eq:bounds_apsibnm} 
\end{align}
which follow directly from~\eqref{eq:bound_psi}, \eqref{eq:bound_Xipsi} 
and~\eqref{eq:bound_a1}. Using~\eqref{eq:bound_bk} and summing over the relevant 
$(n,m)$ yields indeed a bound of order $(1\wedge 
t_+)^{[(\eta+\beta-\abss{k})\wedge0]/\fraks_0}$ for~\eqref{eq:RKQRf2l}, and in 
the case $\fraK=O_T$ we can again extract a factor $T^{\kappa/\fraks_0}$ by 
decreasing $\bar\eta$.

In the case $1\wedge t_+^{1/\fraks_0} < 2^{-n}$, if $\ell < \alpha+\beta$ we
use \eqref{eq:def_bknm} to get
\begin{equation}
 b^{k,+}_{nm}(z) + a^+(z+h_{nm})\hat\chi^k_{\tau,nm}(z) + 
a^+_i(z+h_{nm})\hat\chi^k_{X_i\tau,nm}(z)
 = \bigpscal{\cR^{nm}f_{nm}^+}{\D^k\KQhat_{nm}(z-\cdot)}\;.
\end{equation} 
Here we apply~\cite[Prop.~6.9]{Hairer2014}, which states that
$\cR^{nm}f_{nm}^+\in\cC^\eta_\fraks$, showing that the above quantity has order
$2^{(\abss{k}-\fraks_0-\beta-\eta)n}$. Since by assumption
$\abss{k}-\beta-\eta\neq0$, regardless of its sign, summing over $(n,m)$ yields
a bound of order $(1\wedge t_+)^{[(\eta+\beta-\abss{k})\wedge0]/\fraks_0}$.
If $\alpha+\beta < \ell < \alpha+\beta+1$, we use  
\begin{equation}
 b^{k,+}_{nm}(z) + a^+_i(z+h_{nm})\hat\chi^k_{X_i\tau,nm}(z)
 = \bigpscal{\cR^{nm}f_{nm}^+}{\D^k\KQhat_{nm}(z-\cdot)} 
 - a^+(z+h_{nm})\hat\chi^k_{\tau,nm}(z)\;,
\end{equation} 
so that a bound of the same order follows by combining the two previous
estimates. The case $\ell > \alpha+\beta+1$ is treated similarly.  

It remains to obtain estimates involving two different points $z,\bar z$.
Here we first note that the definition of $\fraK_P$ entering the
definition of $\normDgamma{f}$ implies that if $(z,\bar z)\in\fraK_P$ then $t$
and $\bar t$ necessarily have the same sign and are comparable. Thus we only 
need to consider the case where both $t$ and $\bar t$ are strictly positive, 
and we may drop one of the factors $\Rplus$. 

Lemma~\ref{lem:Gammazz} extends naturally to $\cI^Q_{nm}\tau$ and 
$\hat\chi^k_{nm}(z)$. Proceeding similarly as in Appendix~\ref{app:Gammazz}, 
but in the basis 
$(\set{X^k}_{\abss{k}<\alpha+\beta},\cI^Q_{nm}\tau,\cI^Q_{nm}X_i\tau)$, we 
also obtain 
\begin{alignat}{2}
\nonumber
\Gamma_{z\bar z}\cI^Q_{nm}(X_i\tau) 
={}& 
\cI^Q_{nm}(X_i\tau) + (z_i -\bar z_i)\cI^Q_{nm}\tau \span\span \\
\nonumber
&{}+ 
\sum_{\abss{k}<\alpha+\beta+1} \frac{X^k}{k!} 
\biggl[ &&\hat\chi^k_{X_i\tau,nm}(z) - 
\sum_{\abss{\ell}<\alpha+\beta+1-\abss{k}}\frac{(z-\bar z)^\ell}{\ell!}
\hat\chi^{k+\ell}_{X_i\tau,nm}(\bar z) \\
&&& 
+ (z_i-\bar z_i) \hat\chi^k_{\tau,nm}(z) \indexfct{\abss{k}<\alpha+\beta}
\biggr]\;.
\end{alignat}
This yields the expression 
\begin{align}
\nonumber
\cK^Q_\gamma\Rplus f_2(z) - \Gamma_{z\bar z}\cK^Q_\gamma\Rplus
f_2(\bar z) 
= \sumnm \biggl\{&Q^{\alpha+\beta}_{nm}(z,\bar z)\cI^Q_{nm}\tau 
+ Q^{\alpha+\beta+1}_{nm}(z,\bar z)\cI^Q_{nm}(X_i\tau) \\
&+ \sum_{\abss{k} < \gamma+\beta} \frac{X^k}{k!}Q^k_{nm}(z,\bar z)
\biggr\}\;, 
\end{align} 
where 
\begin{align}
\nonumber
Q^{\alpha+\beta}_{nm}(z,\bar z) ={}& a^+(z+h_{nm}) - a^+(\bar z+h_{nm})
- (z_i-\bar z_i)a^+_i(\bar z+h_{nm})\;, \\
\nonumber
Q^{\alpha+\beta+1}_{nm}(z,\bar z) ={}& a^+_i(z+h_{nm}) 
- a^+_i(\bar z+h_{nm})\;, \\
\nonumber
Q^k_{nm}(z,\bar z) ={}&  
b^{k,+}_{nm}(z) - \sum_{\abss{\ell} < \beta+\gamma-\abss{k}} 
\frac{(z-\bar z)^\ell}{\ell!} b^{k+\ell,+}_{nm}(\bar z) \\
\nonumber
&{}+ Q^{\alpha+\beta}_{nm}(z,\bar z)
\hat\chi^k_{\tau,nm}(z) \indexfct{\abss{k} < \alpha+\beta} \\
&{}+ Q^{\alpha+\beta+1}_{nm}(z,\bar z)
\hat\chi^k_{X_i\tau,nm}(z) \indexfct{\abss{k} < \alpha+\beta+1}
\end{align}
(note that the terms in $a^+(z+h_{nm})\hat\chi^{k+\ell}_{\tau,nm}(\bar z)$
stemming from $\Gamma_{z\bar z}\cI^Q_{nm}$ and $\Gamma_{z\bar z}X^k$ in the
first sum over $k$ in~\eqref{eq:def_KQf2} cancel). 
For the components of non-integer regularity, we obtain
using~\eqref{eq:bound_a1} the required bounds
\begin{align}
\nonumber
 \bigabs{Q^{\alpha+\beta}_{nm}(z,\bar z)} 
&\lesssim
\norm{z-\bar z}_\fraks^{\gamma-\alpha}
(1\wedge t_+\wedge \bar t_+)^{(\eta-\gamma)/\fraks_0}
\normDgamma{\Rplus f_2}_{\gamma,\eta;\bar\fraK}\;, \\
 \bigabs{Q^{\alpha+\beta+1}_{nm}(z,\bar z)} 
&\lesssim 
\norm{z-\bar z}_\fraks^{\gamma-\alpha-1}
(1\wedge t_+\wedge \bar t_+)^{(\eta-\gamma)/\fraks_0}
\normDgamma{\Rplus f_2}_{\gamma,\eta;\bar\fraK}\;,
\end{align} 
from which a factor $T^{\kappa/\fraks_0}$ can be extracted as before. 

Finally, in the case of polynomial terms, we now consider three 
different regimes, depending on the value of $2^{-n}$ compared to 
$\norm{z-\bar z}_\fraks$  and $\frac12(1\wedge t_+\wedge \bar
t_+)^{1/\fraks_0})$ (recall that for $(z,\bar z)\in\fraK_P$ we always have 
$\norm{z-\bar z}_\fraks\leqs(1\wedge t_+\wedge \bar
t_+)^{1/\fraks_0}$). 
In the case $2^{-n} \leqs \norm{z-\bar z}_\fraks$, we again estimate separately
the summands in $Q^k_{nm}(z,\bar z)$, yielding the bound 
\begin{equation}
\label{eq:boundQkmn1} 
 \sumnmarg{2^{-n} \leqs \norm{z-\bar z}_\fraks} 
 \bigabs{Q^k_{nm}(z,\bar z)} 
 \lesssim \norm{z-\bar z}_\fraks^{\beta+\gamma-\abss{k}}
 (1\wedge t_+\wedge \bar t_+)^{(\eta-\gamma)/\fraks_0}
 \normDgamma{\Rplus f_2}_{\gamma,\eta;\bar\fraK}\;.
\end{equation}
For $\norm{z-\bar z}_\fraks < 2^{-n} \leqs \frac12(1\wedge t_+\wedge \bar
t_+)^{1/\fraks_0}$ and $\abss{k} < \alpha+\beta$, we use the fact that 
\begin{equation}
\label{eq:def_QKnm} 
 Q^k_{nm}(z,\bar z) = 
 \bigpscal{\cR^{nm}f_{nm}^+ - \Pi^{nm}_{\bar z}f_{nm}^+(\bar z)}{\hat
K^{Q,k;\gamma}_{nm;\bar zz}}\;, 
\end{equation} 
where $K^{Q,k;\gamma}_{nm;\bar zz}$ is defined as in~\eqref{eq:def_KQalphanm},
but with $\D^k\KQhat_{nm}$ instead of $K^Q_{nm}$. It thus admits the integral
representation 
\begin{equation}
\label{eq:KQkalpha_int} 
  K^{Q,k;\gamma}_{nm;\bar zz}(z') 
 = \sum_{\ell\in\partial A} \int_{\R^{d+1}} \D^{k+\ell}
\KQhat_{nm}(z+h-z')\cQ^\ell(\bar z-z,\6h)
\end{equation} 
where $A=\setsuch{\ell}{\abss{k+\ell} < \gamma+\beta}$. Here we use an argument
similar to the one used in the proof of Lemma~\ref{lem:modelIQ}. Owing to
lack of translation invariance, however, we have to decompose, writing $\tilde
z=z+h$, 
\begin{align}
\nonumber 
  \bigpscal{\cR^{nm}f_{nm}^+ - &\Pi^{nm}_{\bar z}f_{nm}^+(\bar z)}{\D^{k+\ell}
\KQhat_{nm}(\tilde z-\cdot)} \\
\nonumber 
={}& \bigpscal{\cR^{nm}f_{nm}^+ - \Pi^{nm}_{\tilde z}f_{nm}^+(\tilde
z)}{\D^{k+\ell}
\KQhat_{nm}(\tilde z-\cdot)} \\
&{}+ \bigpscal{\Pi^{nm}_{\tilde z}[f_{nm}^+(\tilde z) -
\Gamma_{\tilde z\bar z}f_{nm}^+(\bar z)]}{\D^{k+\ell} \KQhat_{nm}(\tilde
z-\cdot)}\;.
\end{align} 
For the first term on the right-hand side, we apply again the improved
reconstruction theorem~\cite[Lem.~6.7]{Hairer2014} to obtain a bound of order
$(1\wedge\tilde{t}_+)^{\,(\eta-\gamma)/\fraks_0}
2^{(\abss{k+\ell}-\fraks_0-\gamma-\beta)n}$. Since $\cQ^\ell(\bar z-z,\cdot)$ is
supported on values of $h$ such that $\norm{h}_\fraks\leqs\norm{z-\bar z}_\fraks
\leqs \frac12(1\wedge t_+\wedge \bar t_+)^{1/\fraks_0}$, we can replace
$\tilde t_+$ by $t_+\wedge\bar t_+$ in this expression. For the second term, we
use the fact that 
\begin{equation}
 \Gamma_{\tilde z\bar z}f_{nm}^+(\bar z)
= \bigbrak{a^+(\tilde z+h_{nm}) + (\tilde z_i - \bar z_i)a^+(\bar z+h_{nm})}\tau
+ a^+(\bar z+h_{nm})X_i\tau\;, 
\end{equation}
as well as 
\eqref{eq:bound_a1}, \eqref{eq:bound_psi} and~\eqref{eq:bound_Xipsi} to get the 
bound 
\begin{equation}
 \norm{\tilde z-z}_\fraks^{\gamma-\alpha} (1\wedge t_+\wedge \tilde
t_+)^{(\eta-\gamma)/\fraks_0} 2^{(\abss{k+\ell}-\fraks_0-\alpha-\beta)n}\;. 
\end{equation} 
Again, we can replace $\tilde t_+$ by $\bar t_+$, and also bound $\norm{\tilde
z-z}_\fraks^{\gamma-\alpha}$ by $2^{-(\gamma-\alpha)n}$. Adding the last two
bounds and using the fact that $\cQ^\ell(z-\bar z,\cdot)$ has total mass
$\norm{z-\bar z}_\fraks^{\abss{\ell}}$, we obtain 
\begin{equation}
\label{eq:bound_Qknm_01} 
 \bigabs{Q^k_{nm}(z,\bar z)} \lesssim  
 (1\wedge t_+\wedge\bar t_+)^{(\eta-\gamma)/\fraks_0} 
 \sum_{\ell\in\partial A} 
 \norm{z-\bar z}_\fraks^{\abss{\ell}}
 2^{(\abss{k+\ell}-\fraks_0-\gamma-\beta)n}
 \normDgamma{\Rplus f_2}_{\gamma,\eta;\bar\fraK}\;.
\end{equation} 
Summing over the relevant values of $(n,m)$, we again obtain a bound as in the
right-hand side of~\eqref{eq:boundQkmn1}. The same bound holds also in the case
$\abss{k} > \alpha+\beta$, by combining the previous arguments. 

In the last case $2^{-n} > \frac12(1\wedge t_+\wedge \bar t_+)^{1/\fraks_0}$, 
we have the bound 
\begin{equation}
 \bigabs{\bigpscal{\cR^{nm}f_{nm}^+ - \Pi^{nm}_{\bar z}f_{nm}^+(\bar 
z)}{\D^{k+\ell} \KQhat_{nm}(z-\cdot)}}
 \lesssim 2^{(\abss{k+\ell}-\fraks_0-\eta-\beta)n}
 \normDgamma{\Rplus f_2}_{\gamma,\eta;\bar\fraK}\;.
\end{equation} 
Indeed, such a bound holds separately for
$\abs{\pscal{\cR^{nm}f_{nm}^+}{\D^{k+\ell}
\KQhat_{nm}(z-\cdot)}}$ since $\cR^{nm}f_{nm}^+\in\cC^\eta_\fraks$, and for 
$\abs{\pscal{\Pi^{nm}_{\bar z}f_{nm}^+(\bar z)}{\D^{k+\ell}
\KQhat_{nm}(z-\cdot)}}$, as a consequence of~\eqref{eq:bounds_apsibnm} and the 
condition on $2^{-n}$. Substituting in the integral 
expression~\eqref{eq:KQkalpha_int} shows that 
\begin{align}
\nonumber
 \bigabs{Q^k_{nm}&(z,\bar z)} \\
\label{eq:bound_Qknm_02}
 &\lesssim  
 \sum_{\ell\in\partial A} \norm{z-\bar z}_\fraks^{\abss{\ell}} 
 2^{(\abss{k+\ell} - \fraks_0-\eta-\beta)n}
 \normDgamma{\Rplus f_2}_{\gamma,\eta;\bar\fraK} \\
&\lesssim  \norm{z-\bar z}_\fraks^{\beta+\gamma-\abss{k}}
 \sum_{\ell\in\partial A} (1\wedge t_+\wedge \bar
t_+)^{(\abss{k+\ell}-\beta-\gamma)/\fraks_0}
 2^{(\abss{k+\ell} - \fraks_0-\eta-\beta)n}
 \normDgamma{\Rplus f_2}_{\gamma,\eta;\bar\fraK}\;.
\nonumber
\end{align} 
Summing over the relevant values of $(n,m)$ yields again a bound of the
form~\eqref{eq:boundQkmn1}. This completes the proof of the fact that
$\Rplus\cK^Q_\gamma\Rplus f_2\in\cD^{\gamma,\eta}$. As before, a factor
$T^{\kappa/\fraks_0}$ can be extracted when $\fraK = O_T$, which also completes
the proof of~\eqref{eq:schauder2}. 

The proof of~\eqref{eq:schauder3} is very similar to the one just given, using
the estimate~\cite[(3.4)]{Hairer2014} of the reconstruction theorem in place 
of~\cite[(3.3)]{Hairer2014}.

\subsection{Convolution identity}

It remains to prove that the convolution identity~\eqref{eq:schauder1} holds. 
This will follow from the next result, combined with the reconstruction
theorem. 

\begin{lemma}
For every $z=(t,x)\in\R^{d+1}$ such that $t>0$, the bound 
\begin{equation}
 \bigabs{\bigpscal{\Pi_z\cK^Q_\gamma\Rplus f_2(z) - K^Q*\cR \Rplus f_2}
 {\psi^\lambda_z}} \lesssim \lambda^{\gamma+\beta} (1\wedge
t)^{(\eta-\gamma)/\fraks_0}
\label{eq:conv_bound} 
\end{equation} 
holds uniformly in $\lambda\in(0,1\wedge t^{1/\fraks_0}]$. 
\end{lemma}
\begin{proof}
Using the representation 
\begin{equation}
 \bigpscal{K^Q*\cR \Rplus f_2}{\psi^\lambda_z}
 = \sumnm \int_{\R^{d+1}} \bigpscal{\cR^{nm} f_{nm}^+}{\KQhat_{nm}(\bar
z-\cdot)} \psi^\lambda_z(\bar z)\6z
\end{equation} 
and the definition~\eqref{eq:model_IQnm} of the model, we obtain 
\begin{equation}
 \bigpscal{\Pi_z\cK^Q_\gamma\Rplus f_2(z) - K^Q*\cR \Rplus f_2}
 {\psi^\lambda_z} 
 = -\sumnm \int_{\R^{d+1}} Q^0_{nm}(\bar z,z)\psi^\lambda_z(\bar z)\6\bar z\;,
\end{equation} 
with $Q^0_{nm}$ as in~\eqref{eq:def_QKnm}. Since $\psi^\lambda_z$ is
supported in the set $\setsuch{\bar z}{\norm{\bar z-z}_\fraks \leqs \lambda}$,
whenever $\lambda\leqs 2^{-n}\leqs\frac12(1\wedge t\wedge \bar
t\,)^{1/\fraks_0}$, \eqref{eq:bound_Qknm_01} provides the bound 
\begin{equation} 
 \bigabs{Q^0_{nm}(\bar z,z)} \lesssim
 (1\wedge t\wedge \bar t\,)^{(\eta-\gamma)/\fraks_0}
 \sum_{\ell\in\partial A} 
 \norm{z-\bar z}_\fraks^{\abss{\ell}}
 2^{(\abss{\ell}-\fraks_0-\gamma-\beta)n}
 \normDgamma{\Rplus f_2}_{\gamma,\eta;\bar\fraK}
\end{equation}
where $A=\setsuch{\ell}{\abss{\ell}<\gamma+\beta}$. 
For $2^{-n} > \frac12(1\wedge t\wedge \bar t\,)^{1/\fraks_0}$, 
\eqref{eq:bound_Qknm_02} yields 
\begin{equation} 
 \bigabs{Q^0_{nm}(\bar z,z)} \lesssim  
 \norm{z-\bar z}_\fraks^{\beta+\gamma} 
 \sum_{\ell\in\partial A} 
 (1\wedge t\wedge \bar t\,)^{(\abss{\ell}-\gamma-\beta)/\fraks_0}
 2^{(\abss{\ell}-\fraks_0-\eta-\beta)n}
 \normDgamma{\Rplus f_2}_{\gamma,\eta;\bar\fraK}\;.
\end{equation}
Combining this with~\eqref{eq:bound_integral_power}, we obtain 
\begin{equation}
\label{eq:sum_small_n} 
 \sumnmarg{2^{-n}\geqs\lambda} \biggabs{\int_{\R^{d+1}} Q^0_{nm}(\bar
z,z)\psi^\lambda_z(\bar z)\6\bar z} \lesssim \lambda^{\gamma+\beta}
(1\wedge t)^{(\eta-\gamma)/\fraks_0}\;.
\end{equation} 
Finally, in the case $2^{-n} < \lambda$, we use the same argument as
in~\eqref{eq:decomp_YZ}, yielding 
\begin{equation}
 \int_{\R^{d+1}} Q^0_{nm}(\bar z,z) \psi^\lambda_z(\bar z)\6\bar z 
 = \bigpscal{\cR^{nm}f^+_{nm} - \Pi^{nm}_z f^+_{nm}(z)}
 {Y^\lambda_{nm} - \sum_{\abss{\ell} < \gamma+\beta} Z^\lambda_{nm;\ell}}\;.
\end{equation} 
Here we obtain bounds similar to~\eqref{eq:bounds_YZ}, but with $\gamma$
instead of $\alpha$ and an extra factor $(1\wedge t)^{(\eta-\gamma)/\fraks_0}$.
Summing over $m$ and $n$ yields again a bound as on the right-hand side
of~\eqref{eq:sum_small_n}. 
\end{proof}

Combining~\eqref{eq:conv_bound} with~\cite[Lem.~6.7]{Hairer2014}, we obtain
\begin{equation}
  \bigabs{\bigpscal{\cR\cK^Q_\gamma\Rplus f_2(z) - K^Q*\cR \Rplus f_2}
 {\psi^\lambda_z}} \lesssim \lambda^{\gamma+\beta} (1\wedge
t)^{(\eta-\gamma)/\fraks_0}\;,
\end{equation} 
which proves the convolution identity by the uniqueness part of the
reconstruction theorem. 

To complete the proof of Theorem~\ref{thm:Schauder}, we have to show that
$(K^Q*\cR\Rplus f_3)(z) = 0$ for all $z=(t,x)$ such that $t>0$. Here we use
the fact that 
\begin{equation}
 \bigpscal{K^Q*\cR \Rplus f_3}{\psi^\lambda_z}
 = \sumnm \int_{\R^{d+1}} \bigpscal{\cR\Rplus f_3}{\psi^\lambda_z}
 K^Q_{nm}(z-\bar z)\6\bar z\;.
\end{equation} 
Since by the diagonal identity~\eqref{eq:diag_ident}
\begin{equation}
 \lim_{\lambda\to0} \bigpscal{\Pi_z\Rplus f_3(z)}{\psi^\lambda_z} 
 = \lim_{\lambda\to0} \sum_{\tau\in\cF_3} a^+_\tau(z) 
 \bigpscal{\Pi_z \tau}{\psi^\lambda_z} = 0\;,
\end{equation} 
the reconstruction theorem implies that 
$\pscal{\cR \Rplus f_3}{\psi^\lambda_z}$ converges to $0$ as well, and the
desired conclusion follows. 
\qed


\section{Proofs for Section~\ref{sec:renormalisation}}
\label{app:renorm} 


\subsection{Proof of Lemma~\ref{lem:I00}}
\label{app:renorm1} 

Applying Proposition~\ref{prop:KQnm} with $\abs{\fraks}=5$, $\fraks_0=2$ and 
$\beta=2$, we find that $K^Q_{nm}(t,x)$ is supported in the ball  
$\norm{(t,x)-(m2^{-2n},0)}_\fraks \leqs (1+\sqrt{2})2^{-n}$, and is of order 
$2^n$. Using the bound on $\cQ^\eps_0(z)$ given in~\cite[Lem.~6.2]{BK2016}, it 
follows that 
\begin{align}
 \bigabs{I^Q_{00;mn}(\eps)} 
& \lesssim \int_{(m-1)2^{-2n}}^{(m+1)2^{-2n}} 
\int_{\norm{x}\lesssim 2^{-n}} 2^n\frac{1}{\abs{t}+\norm{x}^2+\eps^2} \6x \6t \\
& \lesssim 2^n \int_{(m-1)2^{-2n}}^{(m+1)2^{-2n}} 
\int_0^{2^{-n}}\frac{r^2\6r}{\abs{t}+r^2+\eps^2} \6t\;,
\end{align} 
where we have used polar coordinates, and the equivalence of the $\ell^1$ norm 
$\abs{x}$ and the Euclidean norm $\norm{x}$. For $m<2$, we 
obtain a bound of order $2^{-2n}$ by bounding $\abs{t}+r^2+\eps^2$ below by 
$r^2$. For $m\geqs2$, bounding $\abs{t}+r^2+\eps^2$ below by 
$r^2 + (m-1)2^{-2n}$ yields a bound of order $2^{-2n}m^{-1} \leqs 
2^{1-2n}(m+2)^{-1}$. 
\qed


\subsection{Proof of Proposition~\ref{prop:Fhat}}
\label{app:renorm2} 

It suffices to apply the renormalisation map $M_\eps$ to all monomials in $U$ 
and $V$ of degree $2$ and $3$, when $U$ and $V$ are given 
by~\eqref{eq:UV_expansion}. Using the expressions~\cite[(6.12)]{BK2016} for the 
action of $M_\eps$, one obtains 
\begin{align}
\nonumber
M_\eps U^2 &= U^2 - C_1(\eps)\unit\;, \\
\nonumber
M_\eps U^3 &= U^3 - 3\bigbrak{\varphi C_1(\eps) + b_1 
C_2(\eps)}\unit - 3\bigbrak{C_1(\eps)+3\gamma_1 C_2(\eps)} \RSI
- 9\gamma_2 C_2(\eps) \RSoI 
+ \varrho_{U^3}(U,V)\;, \\
M_\eps U^2 V &= U^2 V - \psi C_1(\eps) \unit - C_1(\eps) \RSoI 
+ \varrho_{U^2V}(U,V)\;, 
\end{align}
where $\varrho_{U^3}(U,V)$ and $\varrho_{U^2V}(U,V)$ are remainder terms of 
strictly positive homogeneity. All other monomials are invariant under $M_\eps$ 
up to remainders of strictly positive homogeneity. The result follows, using 
the expression~\eqref{eq:bc} for $b_1$ with $p=\varphi$ and $q=\psi$, and 
the expansion~\eqref{eq:UV_expansion} in order to express $\RSI$ and $\RSoI$ in 
terms of $U$ and $V$.
\qed


{\small
\bibliography{../BK}

\def\cprime{$'$}
\begin{thebibliography}{1}

\bibitem{BK2016}
N.~Berglund and C.~Kuehn.
\newblock Regularity structures and renormalisation of {F}itz{H}ugh--{N}agumo
  {SPDE}s in three space dimensions.
\newblock {\em Electron. J. Probab.}, 21:1--48, 2016.

\bibitem{Hairer2014}
M.~Hairer.
\newblock A theory of regularity structures.
\newblock {\em Invent. Math.}, 198(2):269--504, 2014.

\end{thebibliography}
\bibliographystyle{abbrv}}



\goodbreak

\vfill

\bigskip\bigskip\noindent
{\small 
Nils Berglund \\
Institut Denis Poisson (IDP) \\ 
Universit\'e d'Orl\'eans, Universit\'e de Tours, CNRS -- UMR 7013 \\
B\^atiment de Math\'ematiques, B.P. 6759\\
45067~Orl\'eans Cedex 2, France \\
{\it E-mail address: }{\tt nils.berglund@univ-orleans.fr}
}

\bigskip\noindent
{\small 
Christian Kuehn \\
Technical University of Munich (TUM) \\
Faculty of Mathematics \\
Boltzmannstr. 3 \\
85748 Garching bei M\"unchen, Germany \\
{\it E-mail address: }{\tt ckuehn@ma.tum.de}
}

\end{document}